\documentclass[11pt,a4paper]{article}
\usepackage{booktabs}
\usepackage{caption}
\usepackage{mathrsfs}
\usepackage{amsmath}
\usepackage{amsfonts,amsthm,amssymb,mathrsfs,bbding}
\usepackage{float}
\usepackage{graphics,epstopdf}
\usepackage{graphics,multicol}
\usepackage{graphicx}
\usepackage{color}
\usepackage{enumerate}
\usepackage{caption}
\usepackage{rotating}
\usepackage{lscape}
\usepackage{longtable}
\allowdisplaybreaks[4]
\usepackage{tabularx}
\usepackage[colorlinks=true,anchorcolor=blue,filecolor=blue,linkcolor=blue,urlcolor=blue,citecolor=blue]{hyperref}
\usepackage{extarrows}
\usepackage{cite}
\usepackage{latexsym,bm}
\usepackage{mathtools}

\usepackage[usenames,dvipsnames]{pstricks}
 \usepackage{pstricks-add}
 \usepackage{epsfig}
 \usepackage{pst-grad} 
 \usepackage{pst-plot} 
 \usepackage[space]{grffile} 
 \usepackage{etoolbox} 

\evensidemargin=\oddsidemargin\topmargin=-1.5cm

\newtheorem{thm}{Theorem}[section]
\newtheorem{prob}{Problem}[section]

\newtheorem{lem}{Lemma}[section]

\newtheorem{conj}{Conjecture}[section]
\newtheorem{claim}{Claim}[section]
\newtheorem{definition}{Definition}[section]

\addtocounter{section}{0}
\begin{document}
\title{Spectral extrema of $K_{s,t}$-minor free graphs--On a conjecture of M. Tait\footnote{Supported by the National Natural Science Foundation of China (Nos. 11971445, 11771141 and 12011530064).}}
\author{{\bf Mingqing Zhai$^{a}$}
, {\bf Huiqiu Lin$^{b}$}\thanks{Corresponding author. E-mail addresses: mqzhai@chzu.edu.cn
(M. Zhai); huiqiulin@126.com (H. Lin).}
\\
{\footnotesize $^a$ School of Mathematics and Finance, Chuzhou University, Chuzhou, Anhui 239012, China} \\
{\footnotesize $^b$ Department of Mathematics, East China University of Science}\\
{\footnotesize  and Technology, Shanghai 200237, China}}
\date{}

\date{}
\maketitle
{\flushleft\large\bf Abstract}
Minors play an important role in extremal graph theory and spectral extremal graph theory.
Tait [The Colin de Verdi\`{e}re parameter, excluded minors, and the spectral radius,
J. Combin. Theory Ser. A 166 (2019) 42--58] determined the maximum spectral radius and characterized the unique extremal graph
for $K_r$-minor free graphs of sufficiently large order $n$,
he also made great progress on
$K_{s,t}$-minor free graphs and posed a conjecture: Let $2\leq s\leq t$ and $n-s+1=pt+q$,
where $n$ is sufficiently large and $1\leq q\leq t.$
Then $K_{s-1}\nabla (pK_t\cup K_q)$ is the unique extremal graph with the maximum spectral radius
over all $n$-vertex $K_{s,t}$-minor free graphs.
In this paper, Tait's conjecture is completely solved.
We also determine the maximum spectral radius and its extremal graphs for $n$-vertex $K_{1,t}$-minor free graphs.
To prove our results, some spectral and structural tools, such as,
local edge maximality, local degree sequence majorization,
double eigenvectors transformation, are used to deduce structural properties of extremal graphs.
\begin{flushleft}
\end{flushleft}

\section{Introduction}

~~~~Given a graph $H$, a graph is said to be \emph{$H$-free} if it does not contain $H$ as a subgraph.
The classic \emph{Tur\'{a}n's problem} asks what is the maximum size of an $H$-free graph of order $n$,
where the maximum size is known as the \emph{Tur\'{a}n number} of $H$ and denoted by $ex(n,H)$.
The study of Tur\'{a}n's problem can be dated back at least to Mantel \cite{MAN} in 1907,
who showed that $ex(n,K_3)\leq \lfloor n^2/4\rfloor$.
Mantel's theorem was extended by Tur\'{a}n's theorem in 1941 \cite{Turan}.
Since then, Tur\'{a}n's problem and many kinds of its variations have been paid much attention
and a considerable number of influential results in extremal graph theory have been obtained
(see for example, a survey, \cite{ZD2}).
In contrast, the spectral extremal problem asks: given a graph $H$,
what is the maximum spectral radius of an $H$-free graph of order $n$?
In the past decades much research has been done on spectral extremal graph theory,
see $K_r$ \cite{BV,WILF}, $K_{s,t}$ \cite{BABA,NIKI3}, $M_k$ \cite{Feng}, $C_{k,q}$ \cite{LP},
$P_k$ \cite{NIKI2}, $F_k$ \cite{CB}, $W_{2k+1}$ \cite{CB1}, $\bigcup_{i=1}^kS_{a_i}$ \cite{CLZ1},
$\bigcup_{i=1}^kP_{a_i}$ \cite{CLZ}, $C_4$ \cite{V1,ZHAI}, $C_6$ \cite{ZL},
consecutive cycles \cite{Hou,LB,V9,NP,ZL1} and a survey \cite{NIKI1}.

Given two graphs $H$ and $G$, $H$ is a \emph{minor} of $G$ if $H$ can be obtained from a
subgraph of $G$ by contracting edges. A graph is said to be \emph{H-minor free},
if it does not contain $H$ as a minor.
Let $A(G)$ be the adjacency matrix of $G$ and $\rho(G)$ be its spectral radius.
Recently, Nikiforov \cite{NIKI} and Tait \cite{Tait1,Tait}
studied the following spectral extremal problem.

\begin{prob}\label{pro3}
Given a graph $H$ or a family $\mathbb{H}$,
what is the maximum spectral radius of an $H$-minor ($\mathbb{H}$-minor) free graph of order $n$?
\end{prob}

Problem \ref{pro3} was initially paid attention in 1990.
Cvetkovi\'{c} and Rowlinson \cite{CP} conjectured
that $\rho(G)\leq \rho(K_1\nabla P_{n-1})$ for any outerplanar graph $G$
with equality if and only if $G\cong K_1\nabla P_{n-1}$.
Boots and Royle \cite{BOOT} and independently Cao and Vince \cite{CAO} conjectured
that $\rho(G)\leq \rho(K_2\nabla P_{n-2})$ for any planar graph $G$ of order $n\geq9$
with equality if and only if $G\cong K_2\nabla P_{n-2}$.
Subsequently, many scholars contributed to these two conjectures
(see \cite{CAO,Hong1,Hong2,SHU}).
Ellingham and Zha \cite{EZ} showed that $\rho(G)\leq2+\sqrt{2n-6}$ for a planar graph $G$.
Dvo\v{r}\'{a}k and Mohar \cite{DM} proved that $\rho(G)\leq \sqrt{8\Delta-16}+3.47$
for a planar graph $G$ with maximum degree $\Delta$.
In 2017, Tait and Tobin \cite{Tait1} confirmed these two old conjectures for sufficiently large $n$.
Recently, Lin and Ning \cite{LIN} confirmed Cvetkovi\'{c}-Rowlinson conjecture completely.
In 2004, Hong \cite{Hong3}
proved that $K_3\nabla (n-3)K_1$ uniquely attains the maximum spectral radius over all $K_5$-minor free graphs.
Tait \cite{Tait} extended Hong's result to $K_r$-minor free graphs
by showing the unique extremal graph is $K_{r-2}\nabla (n-r+2)K_1$.
In 2017, Nikiforov \cite{NIKI} contributed to $K_{2,t}$-minor free graphs,
and the result was extended to $K_{s,t}$-minor free graphs by Tait as shown in the following theorem.

\begin{thm}\label{thm2}\cite{Tait}
Let $2\leq s\leq t$, $n$ be large enough and $G$ be an $n$-vertex $K_{s,t}$-minor free graph.
Then $$\rho(G)\leq\frac12\left(s+t-3+\sqrt{(s+t-3)^2+4(s-1)(n-s+1)-4(s-2)(t-1)}\right),$$
with equality if and only if $t \mid n-s+1$ and $G\cong K_{s-1}\nabla\frac{n-s+1}t K_t$.
\end{thm}

It should be noted that,
if $t\nmid n-s+1$ then the maximum spectral radius together
with its extremal graph is still unknown for $K_{s,t}$-minor free graphs.
To this end, Tait posed the following conjecture.

\begin{conj}\label{conj1}\cite{Tait}
Let $2\leq s\leq t$, $n$ be large enough and $n-s+1=pt+q$, where $1\leq q\leq t$.
Then, the maximum spectral radius of $n$-vertex $K_{s,t}$-minor free graphs
is attained by the join of $K_{s-1}$ with $p$ copies of $K_t$ and a copy of $K_q$.
\end{conj}

Up to now, Conjecture \ref{conj1} has been confirmed for $s+t=4$ \cite{V1,ZHAI}; $s+t=5$ \cite{NIKI};
$s+t=6$ \cite{Wang}; and $q=t$ (see Theorem \ref{thm2}).
For a graph $G$, let $\overline{G}$ be its complement and $S^{k}(G)$ be a graph obtained by subdividing $k$ times of
an edge $uv$ with the minimum degree sum $d_G(u)+d_G(v)$.
Let $H^\star$ be the Petersen graph, and $H_{s,t}$ be a star forest of order $t+1$, precisely,
the disjoint union of $\lfloor\frac{t+1}{s+1}\rfloor$ stars
in which all but at most one are isomorphic to $K_{1,s}$.
In this paper, Conjecture \ref{conj1} is solved.

\begin{thm}\label{thm3}
Let $2\leq s\leq t$, $n-s+1=pt+q$ and $\beta=\lfloor\frac{t+1}{s+1}\rfloor$, where $n$ is large enough and $1\leq q\leq t$.
Let $G^\star$ attain the maximum spectral radius over all $n$-vertex $K_{s,t}$-minor free graphs.
Then
$$ G^\star\cong \left\{
\begin{aligned}
   &K_{s-1}\nabla \left((p-1)K_t \cup \overline{H^\star}\right) &&\hbox{if $q=2$, $t=8$ and $\beta=1$}; \\
   &K_{s-1}\nabla \left((p-1)K_t \cup S^1\left(\overline{H_{s,t}}\right)\right) &&\hbox{if $q=\beta=2$}; \\
   &K_{s-1}\nabla \left((p-q)K_t \cup q\overline{H_{s,t}}\right) &&\hbox{if $q\leq 2(\beta-1)$ except $q=\beta=2$}; \\
   &K_{s-1}\nabla \left(pK_t \cup K_q\right) &&\hbox{otherwise}.
\end{aligned}
\right.
$$
\end{thm}

It remains $K_{1,t}$-minor free graphs.
Let us first consider connected case. A nice result, due to
Ding, Johnson and Seymour \cite{DJS},
determined the maximum size and constructed its extremal graphs for connected $K_{1,t}$-minor free graphs.
However, it seems difficult to completely characterize the extremal graphs.
Inspired by Ding, Johnson and Seymour, we obtain the following spectral extremal result.

\begin{thm}\label{thm4}
Let $t\geq 3$, and $G^*$ attain
the maximum spectral radius over all $n$-vertex connected $K_{1,t}$-minor free graphs.
Then
$$ G^*\cong \left\{
\begin{aligned}
&\overline{H_{1,t}} &&\hbox{if $n=t+1$}; \\
&  S^{n-t}(K_t) &&\hbox{if $n\geq t+2$.}
\end{aligned}
\right.$$
\end{thm}

By the connected case in Theorem \ref{thm4}, we further solve general case.

\begin{thm}\label{thm5}
Let $n\geq t\geq 1$, and $G$ be an $n$-vertex $K_{1,t}$-minor free graph.
Then $\rho(G)\leq t-1$, with equality if and only if
$G$ contains a component either isomorphic to $K_t$,
or isomorphic to $K_{t+1}$ by deleting $\frac {t+1}2$ independent edges.
\end{thm}

Combining with above results,
the spectral extremal problem on $K_{s,t}$-minor free graphs is completely solved for large enough $n$.
To prove our results, we use some spectral and structural tools,
such as, local edge maximality (see Lemma \ref{lem2.4}),
local degree sequence majorization (see Lemma \ref{lem2.8})
and double eigenvectors transformation to deduce structural properties of extremal graphs.


\section{Preliminaries}

~~~~As usual, $V(G)$ is the vertex set and $E(G)$ is the edge set of a graph $G$.
The number of vertices and edges of $G$ are called its \emph{order} and \emph{size},
and denoted by $|G|$ and $e(G)$, respectively.
Given $u\in V(G)$ and a subgraph $H\subseteq G$ (possibly $u\notin V(H)$),
$N_{V(H)}(u)$ is the set of neighbors of $u$ in $V(H)$ and $d_{V(H)}(u)$ is its cardinality.
If $u\in V(H)$, we also use $N_{H}(u)$ and $d_H(u)$ for convenience.
If $S\subseteq V(G)$, then $G[S]$ and $G-S$ stand for the subgraphs of $G$ induced by $S$
and $V(G)\setminus S$, respectively.
If $S\subseteq E(G)$, then $G-S$ denotes the subgraph obtained from deleting all edges in $S$.
If $A,B\subseteq V(G)$, then $e_G(A,B)$ denotes the number of edges with one endpoint in $A$ and the other in $B$,
and particularly, $e_G(A,A)$ is simplified by $e_G(A)$.

Throughout this section, let $s,t,n$ be integers with $2\leq s\leq t$ and $n$ sufficiently large,
$G^\star$ be the extremal graph with the maximum spectral radius $\rho$
over all $n$-vertex $K_{s,t}$-minor free graphs,
and $X=(x_1,x_2,\ldots,x_n)^T$ be the Perron vector of $G^\star$.
Now let us introduce some important lemmas.
The first is due to Chudnovsky, Reed and Seymour.

\begin{lem}\cite{DJS,CRS}\label{lem2.2}
Let $t\geq3$ and $n\geq t+2$. If $G$ is an $n$-vertex connected graph
with no $K_{1,t}$-minor, then $e(G)\leq {t\choose 2}+n-t$, and this is best possible for all $n,t$.
\end{lem}

For $s=2$, Nikiforov \cite{NIKI} proved that $G^\star$ contains a dominating vertex.
Tait showed the following result for general $s$.
This gives a very important information for $G^\star$.

\begin{lem}\cite{Tait}\label{lem2.1}
$G^\star$ contains a clique dominating set $K$
with $|K|=s-1$.
\end{lem}

By Lemma \ref{lem2.1},
we can observe that $G^\star$ is $K_{s,t}$-minor free
if and only if $G^\star-K$ satisfies the following property:
$K_{a,b}$-minor free for all positive integers $a,b$ with $a+b=t+1$ and $1\leq a\leq \min\{s,\lfloor\frac{t+1}{2}\rfloor\}$.
For convenience, we call it \emph{$(s,t)$-property}.

Many known results indicate that an extremal graph with the maximum size
usually is not an extremal graph with the maximum spectral radius.
This observation also happens on $K_{s,t}$-minor free graphs for general $s$ and $t$.
However, the following lemma implies that $G^\star$ has a \emph{local edge maximality}.

\begin{lem}\label{lem2.4}
Let $H$ be a disjoint union of several components of $G^\star-K$ such that $|H|\leq N$ (a constant).
If $H'$ also has $(s,t)$-property with $V(H')=V(H)$, then $e(H')\leq e(H)$.
\end{lem}

\begin{proof}
For convenience, let
$$X_0=\sum_{v\in K}x_v,~~~x_1=\max_{v\in V(H)}x_v~~~\hbox{and}~~~x_2=\min_{v\in V(H)}x_v.$$
Since $H$ has $(s,t)$-property, then $H$ is $K_{1,t}$-minor free
and hence $\Delta(H)<t$. So, $\rho x_1<X_0+tx_1$ and $\rho x_2\geq X_0$.
It follows that
\begin{align}\label{a1}
x_1<\frac{X_0}{\rho-t}~~~\hbox{and}~~~x_2\geq \frac{X_0}{\rho}.
\end{align}
We now give a claim, which will be frequently used in the subsequent proof.
\begin{claim}\label{cl2.1}
Let $a,b$ be two constants with $a>b$. Then $ax_2>bx_1$ and $ax_2^2>bx_1^2$.
\end{claim}

\begin{proof}
By Lemma \ref{lem2.1},
$\Delta(G^\star)=n-1$ and thus $\rho\geq \rho(K_{1,n-1})=\sqrt{n-1}.$
Since $n$ is large enough and $a,b,t$ are constants,
we can easily have
$$ax_2-bx_1>X_0\left(\frac a \rho-\frac b {\rho-t}\right)>0,$$
and similarly, $ax_2^2>bx_1^2$.
\end{proof}

Let $G'=G^\star-E(H)+E(H')$ and $\rho'=\rho(G')$.
Then $G'$ is also $K_{s,t}$-minor free.
By the way of contradiction, assume that $e(H')\geq e(H)+1$.
Then by Claim \ref{cl2.1},
\begin{eqnarray*}
\rho'-\rho &\geq& X^T(A(G')-A(G^\star))X
=2\sum\limits_{uv\in E(H')}x_ux_v-2\sum\limits_{uv\in E(H)}x_ux_v\\
&\geq& 2e(H')x_2^2-2e(H)x_1^2\\
&>& 0,
\end{eqnarray*}
a contradiction with the maximality of $\rho(G^\star)$. So $e(H')\leq e(H)$.
\end{proof}

For a graph $H$ and two vertices $u,v\in V(H)$ (possibly, $uv\notin E(H)$),
we use $d_H(uv)$ to denote $d_H(u)+d_H(v)$ for convenience.

\begin{lem}\label{lem2.5}
Let $H$ be a disjoint union of several components of $G^\star-K$ such that $|H|\leq N$ (a constant),
If $H'$ also has $(s,t)$-property with $V(H')=V(H)$ and $e(H')=e(H)$,
then $$\sum\limits_{uv\in E(H')}d_H(uv)\leq \sum\limits_{uv\in E(H)}d_H(uv),$$
and if equality holds, then $$\sum\limits_{uv\in E(H')}d_{H'}(uv)\leq \sum\limits_{uv\in E(H)}d_{H}(uv).$$
\end{lem}

\begin{proof}
Let $G'=G^\star-E(H)+E(H')$ and $\rho'=\rho(G')$.
Then $\rho\geq \rho'$ and
\begin{align}\label{a2}
\frac{1}{2}\rho^2(\rho'-\rho)\geq \frac{1}{2}\rho^2X^T(A(G')-A(G^\star))X
=\sum\limits_{uv\in E(H')}\rho x_u\rho x_v-\sum\limits_{uv\in E(H)}\rho x_u\rho x_v.
\end{align}
Recall that $$X_0=\sum_{v\in K}x_v,~~~x_1=\max_{v\in V(H)}x_v~~~\hbox{and}~~~x_2=\min_{v\in V(H)}x_v.$$
Thus, for any $v\in V(H)$, we have
\begin{align}\label{a3}
X_0+d_H(v)x_2\leq\rho x_v\leq X_0+d_H(v)x_1.
\end{align}
It follows that $$\sum\limits_{uv\in E(H')}\rho x_u\rho x_v\geq\sum\limits_{uv\in E(H')}(X_0+d_H(u)x_2)(X_0+d_H(v)x_2),$$
$$\sum\limits_{uv\in E(H)}\rho x_u\rho x_v\leq \sum\limits_{uv\in E(H)}(X_0+d_H(u)x_1)(X_0+d_H(v)x_1).$$
Now let
$$a=\sum\limits_{uv\in E(H')}d_H(uv),~~b=\sum_{uv\in E(H)}d_H(uv),~~c=\sum\limits_{uv\in E(H)}d_H(u)d_H(v),$$
and suppose to the contrary that $a\geq b+1.$ Then
\begin{eqnarray*}
\sum\limits_{uv\in E(H')}\rho x_u\rho x_v-\sum\limits_{uv\in E(H)}\rho x_u\rho x_v &\geq& aX_0x_2-bX_0x_1-cx_1^2\\
&=&X_0\left(\left(a-\frac12\right)x_2-bx_1\right)+\left(\frac12 X_0x_2-cx_1^2\right).
\end{eqnarray*}
By Claim \ref{cl2.1}, $\left(a-\frac12\right)x_2>bx_1.$
And by (\ref{a1}),
$$\frac12 X_0x_2-cx_1^2>X^2_0\left(\frac{1}{2\rho}-\frac{c}{(\rho-t)^2}\right)>0,$$
since $\rho\geq\sqrt{n-1}$ and $n$ is large enough.
Combining with (\ref{a2}), we have $\rho'>\rho$,
a contradiction. Hence, the first inequality holds.

We now prove the second inequality.
Let $a=b$ and $Y=(y_1,y_2,\ldots,y_n)^T$ be the Perron vector of $G'$ such that
$y_i$ and $x_i$ correspond the the same vertex $i$.
Suppose to the contrary that $a'\geq b+1$,
and assume that
$$Y_0=\sum_{v\in K}y_v,~~~y_1=\max_{v\in V(H')}y_v~~~\hbox{and}~~~y_2=\min_{v\in V(H')}y_v.$$
Similar with (\ref{a3}), we have
\begin{align}\label{a4}
Y_0+d_{H'}(v)y_2\leq\rho' y_v\leq Y_0+d_{H'}(v)y_1.
\end{align}
for any $v\in V(H')$. It follows that
\begin{align}\label{a5}
y_1<\frac{Y_0}{\rho'-t}~~~\hbox{and}~~~y_2\geq \frac{Y_0}{\rho'}.
\end{align}
Now we can see that
\begin{eqnarray*}
\rho'\rho(\rho'-\rho)Y^TX &=& \rho'\rho\left((A(G')Y)^TX-Y^T(A(G)X)\right)\\
&=& \sum\limits_{uv\in E(H')}(\rho x_u\rho'y_v+\rho x_v\rho'y_u)-\sum\limits_{uv\in E(H)}(\rho x_u\rho'y_v+\rho x_v\rho'y_u).
\end{eqnarray*}
Now let $$a'=\sum_{uv\in E(H')}d_{H'}(uv),b'=\sum_{uv\in E(H)}d_{H'}(uv),
c'=\sum\limits_{uv\in E(H)}\left(d_H(u)d_{H'}(v)+d_{H}(v)d_{H'}(u)\right).$$
Note that $e(H)=e(H')$. By (\ref{a3}) and (\ref{a4}), we have
\begin{eqnarray*}
&&\rho'\rho(\rho'-\rho)Y^TX\\
&\geq&\sum\limits_{uv\in E(H')}\left((X_0+d_H(u)x_2)(Y_0+d_{H'}(v)y_2)+(X_0+d_H(v)x_2)(Y_0+d_{H'}(u)y_2)\right)\\
&&-\sum\limits_{uv\in E(H)}\left((X_0+d_H(u)x_1)(Y_0+d_{H'}(v)y_1)+(X_0+d_H(v)x_1)(Y_0+d_{H'}(u)y_1)\right)\\
&\geq& aY_0x_2+a'X_0y_2-bY_0x_1-b'X_0y_1-c'x_1y_1.
\end{eqnarray*}
Recall that $d_H(uv)=d_H(u)+d_H(v)$. We can observe that
$$\sum_{uv\in E(H')}d_{H}(uv)=\sum_{v\in V(H')}d_H(v)d_{H'}(v)=\sum_{v\in V(H)}d_H(v)d_{H'}(v)=\sum_{uv\in E(H)}d_{H'}(uv),$$
that is, $a=b'.$ Combining with $a=b$ and $a'\geq a+1$, we have
\begin{eqnarray*}
\rho'\rho(\rho'-\rho)Y^TX &\geq& aY_0x_2+(a+1)X_0y_2-aY_0x_1-aX_0y_1-c'x_1y_1\\
&=& X_0\left(\left(a+\frac12\right)y_2-ay_1\right)+\left(\frac{X_0y_2}2+aY_0\left(x_2-x_1\right)-c'x_1y_1\right).
\end{eqnarray*}
By Claim \ref{cl2.1}, we have
$\left(a+\frac12\right)y_2>ay_1.$
And by (\ref{a1}) and (\ref{a5}), we have
$$\frac {X_0y_2}2+aY_0\left(x_2-x_1\right)-c'x_1y_1>
X_0Y_0\left(\frac1{2\rho'}+\frac a\rho-\frac a{\rho-t}-\frac{c'}{(\rho-t)(\rho'-t)}\right)>0,$$
since $\rho\geq\rho'\geq \rho(K_{1,n-1})=\sqrt{n-1}$ and $a,c',t$ are constants.
It follows that $\rho'\rho(\rho'-\rho)Y^TX>0$, that is, $\rho'>\rho$, a contradiction.
This completes the proof.
\end{proof}

\begin{definition}
\emph{Let $X=(x_1,x_2,\ldots,x_n)^T$ and $Y=(y_1,y_2,\ldots,y_n)^T$ be two decreasing real vectors.
If $$\sum_{i=1}^kx_{i}\le \sum_{i=1}^ky_{i},~~k=1,2,\ldots,n,$$
then we say that $X$ is} weakly majorized \emph{by
$Y$ and denote it by $X\prec_w Y$. If $X\prec_w Y$
and $\sum_{i=1}^nx_i=\sum_{i=1}^ny_i$, then we say
that $X$ is} majorized \emph{by $Y$ and denote it by
$x \prec y$.}
\end{definition}

The following two lemmas are needed in the proof of Lemma \ref{lem2.8},
which is one of our main tools in this paper.

\begin{lem}\cite{LIN}\label{lem2.6}
Let $X=(x_1,x_2,\ldots,x_n)^T$, $Y=(y_1,y_2,\ldots,y_n)^T$ be two nonnegative
decreasing real vectors. If $X\prec_w Y$,
then $\|X\|_p\leq\|Y\|_p$ for $p>1$, with equality holding if and only if $X=Y$.
\end{lem}

\begin{lem}[Exercises 5 (i), P74, \cite{Zhan}]\label{lem2.7}
Let $X,Y,Z\in R^n$ be three
decreasing vectors.
If $X\prec Y$, then $X^T\cdot Z\leq Y^T\cdot Z$.
\end{lem}

Let $\pi(G)$ be the decreasing sequence of vertex degrees of a graph $G$.
By using majorization, B{\i}y{\i}ko\v{g}lu and Leydold \cite{BL} showed that,
for two trees $T$ and $T'$ with $|T|=|T'|$,
if $\pi(T')\prec\pi(T)$, then $\rho(T')\leq \rho(T).$
Unfortunately, this nice tool does not work for general graphs, even for unicyclic graphs.
However, the following lemma implies that $G^\star$ has a local degree sequence majorization.

\begin{lem}\label{lem2.8}
Let $H$ and $H'$ be defined as in Lemma \ref{lem2.5}.
Let $\pi(H)=(d_1,\ldots,d_{|H|})$ and $\pi(H')=(d'_1,\ldots,d'_{|H|})$.
If $\pi(H)\prec\pi(H')$, then $\pi(H)=\pi(H')$.
\end{lem}

\begin{proof}
Let $V(H)=\{v_1,v_2\ldots,v_{|H|}\}$ with $d_H(v_i)=d_i$.
Since $V(H)=V(H')$, by graph isomorphism, we may let $d_{H'}(v_i)=d'_i$ for $i=1,2,\ldots,|H|$.
Suppose to the contrary that $\pi(H)\neq\pi(H')$. Then by Lemma \ref{lem2.6}, we have
\begin{align}\label{a6}
\sum\limits_{i=1}^{|H|}{d'}_i^2>\sum\limits_{i=1}^{|H|}d_i^2.
\end{align}
Now let $X=Z=\pi(H)$ and $Y=\pi(H')$ in Lemma \ref{lem2.7}. Then we have
\begin{align}\label{a7}
\sum\limits_{i=1}^{|H|}d_i^2\leq\sum\limits_{i=1}^{|H|}d_id'_i.
\end{align}
Furthermore, we see that
$$\sum\limits_{i=1}^{|H|}d_id'_i=\sum\limits_{v\in V(H')}d_H(v)d_{H'}(v)
=\sum\limits_{uv\in E(H')}(d_H(u)+d_{H}(v))=\sum\limits_{uv\in E(H')}d_H(uv).$$
And similarly,
$$\sum\limits_{i=1}^{|H|}d_i^2=\sum\limits_{uv\in E(H)}(d_H(u)+d_H(v))=\sum\limits_{uv\in E(H)}d_H(uv).$$
It follows from the first inequality of Lemma \ref{lem2.5} that
$\sum\limits_{i=1}^{|H|}d_id'_i\leq \sum\limits_{i=1}^{|H|}d_i^2.$
Combining with (\ref{a7}), we have $\sum\limits_{i=1}^{|H|}d_id'_i=\sum\limits_{i=1}^{|H|}d_i^2.$
By the second inequality of Lemma \ref{lem2.5}, we have
$$\sum\limits_{uv\in E(H')}d_{H'}(uv)\leq\sum\limits_{uv\in E(H)}d_{H}(uv),$$
that is,
$\sum\limits_{i=1}^{|H|}{d'}_i^2\leq\sum\limits_{i=1}^{|H|}d_i^2,$
which contradicts (\ref{a6}). So, $\pi(H)=\pi(H')$.
\end{proof}

\section{Characterization of Components in $G^\star-K$}

~~~~
Throughout this section, assume that $t\geq4$ and $2\leq s\leq t$.
Let $G^\star, K, H_{s,t}$ and related notations be defined as above.
In particular, recall that each component $H$ of $G^\star-K$ satisfies $(s,t)$-property,
that is, $H$ is $K_{a,b}$-minor free for all positive integers $a,b$ with $a+b=t+1$ and
$1\leq a\leq \gamma$, where $\gamma=\min\{s,\lfloor\frac{t+1}{2}\rfloor\}$.

\begin{lem}\label{lem3.0}
Let $G$ be a connected graph with $|G|=t+1$.
Then $G$ has $(s,t)$-property if and only if
each component of $\overline{G}$ has at least $\gamma+1$ vertices.
\end{lem}

\begin{proof}
Since $|G|=t+1$, a $K_{a,b}$-minor is equivalent to a copy of $K_{a,b}$
for any positive integers $a,b$ with $a+b=t+1$ and
$1\leq a\leq \gamma$.
It is easy to see that $G$ has $(s,t)$-property if and only if
each component of $\overline{G}$ has at least $\gamma+1$ vertices.
\end{proof}

\begin{lem}\label{lem3.1}
Let $\beta=\lfloor\frac{t+1}{s+1}\rfloor$ and $G$ be a connected graph with $|G|=t+1$.
If $G$ has $(s,t)$-property, then
$e(G)\leq {t\choose 2}+\beta-1$.
\end{lem}

\begin{proof}
By Lemma \ref{lem3.0},
each component of $\overline{G}$ has at least $\gamma+1$ vertices.
Now assume that $G$ is edge-maximal and $\overline{G}$ has $c$ components.
Then each component of $\overline{G}$ is a tree and $c$ is maximal.
This implies that $c=\lfloor\frac{t+1}{\gamma+1}\rfloor$ and $e\left(\overline{G}\right)=|G|-c.$
If $s\leq \lfloor\frac{t+1}{2}\rfloor$, then $\gamma=s$ and thus $c=\beta$.
It follows that
\begin{align}\label{a3.0}
e(G)={{t+1}\choose 2}-e\left(\overline{G}\right)={t\choose 2}+c-1={t\choose 2}+\beta-1.
\end{align}
If $s>\lfloor\frac{t+1}{2}\rfloor$,
then $\gamma=\lfloor\frac{t+1}{2}\rfloor$ and $\beta=1$.
We also have $c=\beta$ and hence (\ref{a3.0}) holds.
\end{proof}

The following lemma holds clearly.

\begin{lem}\label{lem3.2}
Let $G$ be a graph with $vw\in E(G)$ and $uw\notin E(G)$.
If $d_G(u)\geq d_G(v)$, then $\pi(G)\prec\pi(G-\{vw\}+\{uw\})$ and $\pi(G)\neq\pi(G-\{vw\}+\{uw\})$.
\end{lem}

Recall that $\beta=\lfloor\frac{t+1}{s+1}\rfloor$ and
$H_{s,t}=(\beta-1)K_{1,s}\cup K_{1,\alpha}$, where $\alpha=t-(\beta-1)(s+1)\geq s$.
The following theorem presents a clear characterization
of a $(t+1)$-vertex component of $G^\star-K$, which
is a key subgraph of the extremal graphs.

\begin{thm}\label{thm3.1}
Let $H$ be a component of $G^\star-K$. If $|H|=t+1$, then $\beta\geq2$,
$H\cong \overline{H_{s,t}}$ and $e(H)={t\choose 2}+\beta-1$.
\end{thm}

\begin{proof}
We first show $\beta\geq2$. Suppose that $\beta=1$.
Then by Lemma \ref{lem3.1}, $e(H)\leq {t \choose 2}=e(K_t\cup K_1).$
By Lemma \ref{lem2.4}, $e(H)=e(K_t\cup K_1)$,
since $K_t\cup K_1$ also has $(s,t)$-property.
Furthermore, since $\Delta(H)\leq t-1$ and $\delta(H)\geq1$,
we can see that $\pi(H)\prec\pi(K_t\cup K_1)$ and $\pi(H)\neq\pi(K_t\cup K_1)$,
which contradicts Lemma \ref{lem2.8}.
Thus, $\beta\geq2$ and correspondingly $s\leq \lfloor\frac{t-1}2\rfloor$.
It follows that $\gamma=s$.

Now we characterize the structure of $H$.
By Lemma \ref{lem2.4}, $H$ is edge-maximal, and combining with Lemmas \ref{lem3.0} and \ref{lem3.1},
we have $e(H)={t\choose 2}+\beta-1$ and
$\overline{H}$ is a forest with
$\beta$ components, each of which has at least $\gamma+1$ ($=s+1$) vertices.
We divide the proof into two claims.

\begin{claim}\label{cl3.1}
Each component of $\overline{H}$ is a star.
\end{claim}

\begin{proof}
Suppose that $\overline{H}$ contains a component $T$ which is not a star.
Then, $T$ contains at least two pendant edges $u_1v_1$ and $u_2v_2$, where $v_1,v_2$ are leaves and $u_1\neq u_2$.
Assume without loss of generality that $d_{T}(u_1)\leq d_{T}(u_2)$. Then
$d_{H}(u_1)\geq d_{H}(u_2)$.
By Lemma \ref{lem3.2}, $\pi(H)\prec\pi(H-\{u_2v_1\}+\{u_1v_1\}),$
which contradicts Lemma \ref{lem2.8}.
\end{proof}

\begin{claim}\label{cl3.2}
$\overline{H}\cong H_{s,t}.$
\end{claim}

\begin{proof}
It suffices to show that of $\overline{H}$ has at most one component not isomorphic to $K_{1,s}$.
Suppose that there exist two components $T_1$ and $T_2$ with $|T_2|\geq |T_1|\geq s+2$.
Assume that $u_iv_i\in E(T_i)$ and $v_i$ is a leaf for $i\in\{1,2\}$.
Similar to the proof of Claim \ref{cl3.1}, we can see that $\pi(H)\prec\pi(H-\{u_2v_1\}+\{u_1v_1\}),$
a contradiction.
\end{proof}

Having Claim \ref{cl3.2}, it remains to show that $\overline{H_{s,t}}$ has $(s,t)$-property.
Since $\gamma=s$ and
$H_{s,t}\cong K_{1,\alpha}\cup (\beta-1)K_{1,s}$,
each component of $H_{s,t}$ has at least $\gamma+1$ vertices.
By Lemma \ref{lem3.0}, $\overline{H_{s,t}}$ has $(s,t)$-property.
\end{proof}

\begin{figure}[htp]
\setlength{\unitlength}{0.75pt}
\begin{center}
\begin{picture}(379.2,153.7)
\qbezier(35.5,91.7)(35.5,75.9)(32.6,64.8)\qbezier(32.6,64.8)(29.6,53.7)(25.4,53.7)
\qbezier(25.4,53.7)(21.2,53.7)(18.2,64.8)\qbezier(18.2,64.8)(15.2,75.9)(15.2,91.7)
\qbezier(15.2,91.7)(15.2,107.5)(18.2,118.6)\qbezier(18.2,118.6)(21.2,129.8)(25.4,129.8)
\qbezier(25.4,129.8)(29.6,129.8)(32.6,118.6)\qbezier(32.6,118.6)(35.5,107.5)(35.5,91.7)
\qbezier(127.9,63.8)(124.6,52.2)(120.0,52.2)\qbezier(120.0,52.2)(115.3,52.2)(112.0,63.8)
\qbezier(112.0,63.8)(108.8,75.3)(108.8,91.7)\qbezier(108.8,91.7)(108.8,108.1)(112.0,119.7)
\qbezier(112.0,119.7)(115.3,131.2)(120.0,131.2)\qbezier(120.0,131.2)(124.6,131.2)(127.9,119.7)
\qbezier(127.9,119.7)(131.2,108.1)(131.2,91.7)\qbezier(131.2,91.7)(131.2,75.3)(127.9,63.8)
\put(72.5,89.9){\circle*{4}}\put(90.6,63.1){\circle*{4}}
\put(55.8,63.1){\circle*{4}}\put(25.4,112.4){\circle*{4}}
\put(25.4,69.6){\circle*{4}}\put(120.4,112.4){\circle*{4}}
\put(120.4,68.9){\circle*{4}}\put(26.1,95.0){\circle*{2}}
\put(26.1,86.3){\circle*{2}}\put(26.1,90.6){\circle*{2}}
\put(121.1,96.4){\circle*{2}}\put(121.1,87.7){\circle*{2}}
\put(121.1,92.1){\circle*{2}}
\qbezier(25.4,112.4)(48.9,101.1)(72.5,89.9)\qbezier(72.5,89.9)(48.9,79.8)(25.4,69.6)
\qbezier(72.5,89.9)(96.4,101.1)(120.4,112.4)\qbezier(72.5,89.9)(96.4,79.4)(120.4,68.9)
\qbezier(120.4,31.5)(120.4,26.9)(106.5,23.6)\qbezier(106.5,23.6)(92.7,20.3)(73.2,20.3)
\qbezier(73.2,20.3)(53.7,20.3)(39.9,23.6)\qbezier(39.9,23.6)(26.1,26.9)(26.1,31.5)
\qbezier(26.1,31.5)(26.1,36.2)(39.9,39.5)\qbezier(39.9,39.5)(53.7,42.8)(73.2,42.8)
\qbezier(73.2,42.8)(92.7,42.8)(106.5,39.5)\qbezier(106.5,39.5)(120.3,36.2)(120.4,31.5)
\put(47.9,31.9){\circle*{4}}\put(97.2,31.9){\circle*{4}}
\put(68.2,31.2){\circle*{2}}\put(77.6,31.2){\circle*{2}}
\put(72.5,31.2){\circle*{2}}
\qbezier(55.8,63.1)(51.8,47.5)(47.9,31.9)\qbezier(55.8,63.1)(76.5,47.5)(97.2,31.9)
\qbezier(90.6,63.1)(69.2,47.5)(47.9,31.9)\qbezier(90.6,63.1)(93.9,47.5)(97.2,31.9)
\qbezier(72.5,89.9)(81.6,76.5)(90.6,63.1)
\put(224.8,134.1){\circle*{4}}\put(353.8,134.1){\circle*{4}}
\qbezier(224.8,134.1)(289.3,134.1)(353.8,134.1)
\put(289.3,134.1){\circle*{4}}
\qbezier(379.5,47.1)(379.5,36.8)(372.2,29.4)\qbezier(372.2,29.4)(364.9,22.1)(354.5,22.1)
\qbezier(354.5,22.1)(344.2,22.1)(336.8,29.4)\qbezier(336.8,29.4)(329.5,36.8)(329.5,47.1)
\qbezier(329.5,47.1)(329.5,57.5)(336.8,64.8)\qbezier(336.8,64.8)(344.2,72.1)(354.5,72.1)
\qbezier(354.5,72.1)(364.9,72.1)(372.2,64.8)\qbezier(372.2,64.8)(379.5,57.5)(379.5,47.1)
\qbezier(250.1,47.9)(250.1,37.4)(242.7,29.9)\qbezier(242.7,29.9)(235.2,22.5)(224.8,22.5)
\qbezier(224.8,22.5)(214.3,22.5)(206.8,29.9)\qbezier(206.8,29.9)(199.4,37.4)(199.4,47.8)
\qbezier(199.4,47.8)(199.4,58.3)(206.8,65.8)\qbezier(206.8,65.8)(214.3,73.2)(224.8,73.2)
\qbezier(224.8,73.2)(235.2,73.2)(242.7,65.8)\qbezier(242.7,65.8)(250.1,58.3)(250.1,47.9)
\qbezier(314.0,93.5)(314.0,83.3)(306.7,76.1)\qbezier(306.7,76.1)(299.5,68.8)(289.3,68.8)
\qbezier(289.3,68.8)(279.0,68.8)(271.8,76.1)\qbezier(271.8,76.1)(264.6,83.3)(264.6,93.5)
\qbezier(264.6,93.5)(264.6,103.8)(271.8,111.0)\qbezier(271.8,111.0)(279.0,118.2)(289.3,118.2)
\qbezier(289.3,118.2)(299.5,118.2)(306.7,111.0)\qbezier(306.7,111.0)(314.0,103.8)(314.0,93.5)
\qbezier(289.3,134.1)(282.4,120.7)(275.5,107.3)\qbezier(289.3,134.1)(296.2,120.7)(303.1,107.3)
\qbezier(224.8,134.1)(219.3,98.6)(213.9,63.1)\qbezier(224.8,134.1)(229.5,98.6)(234.2,63.1)
\qbezier(353.8,134.1)(349.8,99.3)(345.8,64.5)\qbezier(353.8,134.1)(358.2,99.3)(362.5,64.5)
\qbezier(240.7,60.2)(255.6,77.2)(270.4,94.3)\qbezier(242.9,54.4)(261.7,65.6)(280.6,76.9)
\qbezier(308.9,93.5)(324.1,76.5)(339.3,59.5)\qbezier(300.2,77.6)(317.9,65.3)(335.7,52.9)
\qbezier(239.3,43.5)(288.9,41.7)(338.6,39.9)\qbezier(237.8,31.9)(288.6,33.7)(339.3,35.5)
\put(68.2,105.2){\makebox(0,0)[tl]{$w$}}\put(93.7,60.5){\makebox(0,0)[tl]{$u_1$}}
\put(35.3,60.5){\makebox(0,0)[tl]{$u_2$}}
\put(132.7,100.8){\makebox(0,0)[tl]{$\overline{K}_a$}}\put(-14.0,100.8){\makebox(0,0)[tl]{$\overline{K}_b$}}
\put(-14,40){\makebox(0,0)[tl]{$\overline{K}_c$}}\put(353.5,147.7){\makebox(0,0)[tl]{$w$}}
\put(216.8,147.7){\makebox(0,0)[tl]{$u_1$}}\put(281.3,147.7){\makebox(0,0)[tl]{$u_2$}}
\put(280.5,98.8){\makebox(0,0)[tl]{$K_a$}}\put(217.4,52.6){\makebox(0,0)[tl]{$K_b$}}
\put(348.0,52.6){\makebox(0,0)[tl]{$K_c$}}\put(58,1.0){\makebox(0,0)[tl]{$H_{a,b,c}$}}
\put(278,1.0){\makebox(0,0)[tl]{$\overline{H_{a,b,c}}$}}
\qbezier(25.4,112.4)(40.6,87.7)(55.8,63.1)\qbezier(25.4,69.6)(40.6,66.3)(55.8,63.1)
\qbezier(120.4,112.4)(105.5,87.7)(90.6,63.1)\qbezier(120.4,68.9)(105.5,66.0)(90.6,63.1)
\end{picture}
\end{center}
\caption{The graph $H_{a,b,c}$ and its complement.}\label{fig3.1}
\end{figure}
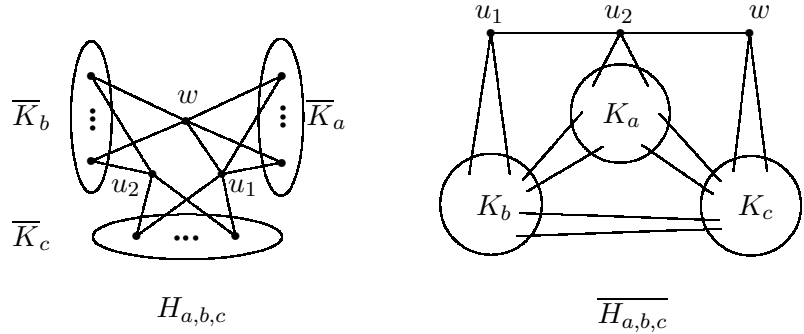

\begin{lem}\label{lem3.3}
Let $\beta\leq2$ and $G$ be a connected graph with $|G|=t+2$. If $G$ has $(s,t)$-property, then
$e(G)\leq{t\choose 2}+2$, and if equality holds,
then $\overline{G}$ is isomorphic to either some $H_{a,b,c}$ (see Fig. \ref{fig3.1})
or the Petersen graph $H^\star$.
\end{lem}

\begin{proof}
Since $G$ satisfies $(s,t)$-property, $G$ contains no $K_{1,t}$-minor.
By Lemma \ref{lem2.2}, we have $e(G)\leq{t\choose 2}+2$.
In the following, assume that $e(G)={t\choose 2}+2$.
Then $e\left(\overline{G}\right)=2t-1.$
The proof is divided into five claims.

\begin{claim}\label{cl3.3}
$\overline{G}$ is connected and $diam(\overline{G})=2$,
where $diam(\overline{G})$ is the diameter of $\overline{G}$.
\end{claim}

\begin{proof}
If $\overline{G}$ is not connected or $diam\left(\overline{G}\right)\geq 3$,
then there exist a pair of non-adjacent vertices $u,v$ with no
common neighbor in $\overline{G}$.
This implies that $uv\in E(G)$ and $N_G(u)\cup N_G(v)=V(G)$.
Since $|G|=t+2$, we can conclude that $G$ contains a $K_{1,t}$-minor,
a contradiction with $(s,t)$-property.
\end{proof}

\begin{claim}\label{cl3.4}
Let $S=\{u_1,\ldots,u_{|S|}\}$ be a minimum cut set of $\overline{G}$.
Then $2\leq|S|\leq3$.
\end{claim}

\begin{proof}
Let $d_G(u,v)$ be the distance of $u,v\in V(G)$.
If $|S|=1$,
then $\overline{G}-\{u_1\}$ has at least two components $G_1$ and $G_2$.
Since $G$ is connected, there exists a vertex $v_1$, say $v_1\in V(G_1)$,
with $u_1v_1\notin E\left(\overline{G}\right)$.
Now, $d_{\overline{G}}(v_1,v)\geq3$ for any $v\in V(G_2)$, contradicting Claim \ref{cl3.3}.
Therefore, $|S|\geq2$. Furthermore,
since $\delta\left(\overline{G}\right)\geq|S|$.
Thus, $2e\left(\overline{G}\right)\geq\delta\left(\overline{G}\right)\cdot|G|\geq|S|(t+2).$
Combining with $e\left(\overline{G}\right)=2t-1$, we have $|S|\leq3.$
\end{proof}

Now let $W=V(G)\setminus S$, $W_1=\{w\in W| |N_{\overline{G}}(w)\cap S|=1\}$
and $\overline{G}[W]$ consists of $k$ components $T_1,\ldots,T_k$, where $k\geq2$.

\begin{claim}\label{cl3.5}
$S$ dominates $W$ and each $V(T_i)$ dominates $S$ in $\overline{G}$.
Moreover, $W_1$ belongs to a single component $T_i$.
\end{claim}

\begin{proof}
Firstly, suppose that $N_{\overline{G}}(w_0)\cap S=\varnothing$ for some $w_0\in W$.
Without loss of generality, assume that $w_0\in V(T_1)$. Then
$d_{\overline{G}}(w_0,w)\geq 3$ for any $w\in V(T_2)$, a contradiction.
Hence, $S$ dominates $W$ in $\overline{G}$.
Secondly, if a $V(T_i)$ does not dominate $S$ in $\overline{G}$,
say $u_1\notin\cup_{w\in V(T_i)}N_{\overline{G}}(w)$, then $S\setminus\{u_1\}$ is a cut set,
which contradicts the minimality of $|S|$.

Now we show that $W_1\subseteq V(T_i)$ for some $i$. Suppose to the contrary that $w_i\in W_1\cap V(T_i)$ for $i=1,2$.
Since $diam(\overline{G})=2$, $w_1$ and $w_2$ have a unique and common neighbor, say $u_1$, in $S$.
It follows that $wu_1\in E\left(\overline{G}\right)$ for any $w\in W$
(Otherwise, either $d_{\overline{G}}(w,w_1)\geq 3$ or $d_{\overline{G}}(w,w_2)\geq 3$).
Therefore, $d_{\overline{G}}(u_1)\geq|W|.$
If $|S|=2$, then
$$e\left(\overline{G}-\{u_1\}\right)\leq e\left(\overline{G}\right)-|W|=(2t-1)-t=t-1.$$
This implies that $\overline{G}-\{u_1\}$ is not connected, which contradicts $|S|=2$.
If $|S|=3$, then $d_{\overline{G}}(u_2)\geq \delta\left(\overline{G}\right)\geq|S|=3.$
It follows that $$e\left(\overline{G}-\{u_1,u_2\}\right)\leq e\left(\overline{G}\right)-(|W|+2)=(2t-1)-(t+1)=t-2.$$
Then $\{u_1,u_2\}$ is a cut set of $\overline{G}$, contradicting $|S|=3$.
So, the claim holds.
\end{proof}

Now we may assume without loss of generality that $W_1\subseteq V(T_1).$
Note that \begin{align}\label{a3.1}
e\left(\overline{G}\right)=
e_{\overline{G}}(S)+\sum_{i=1}^ke_{\overline{G}}(V(T_i))+\sum_{i=1}^ke_{\overline{G}}(V(T_i),S).
\end{align}

\begin{claim}\label{cl3.6}
If $|S|=2$, then $\overline{G}$ is isomorphic to some $H_{a,b,c}$ (see Fig. \ref{fig3.1}).
\end{claim}

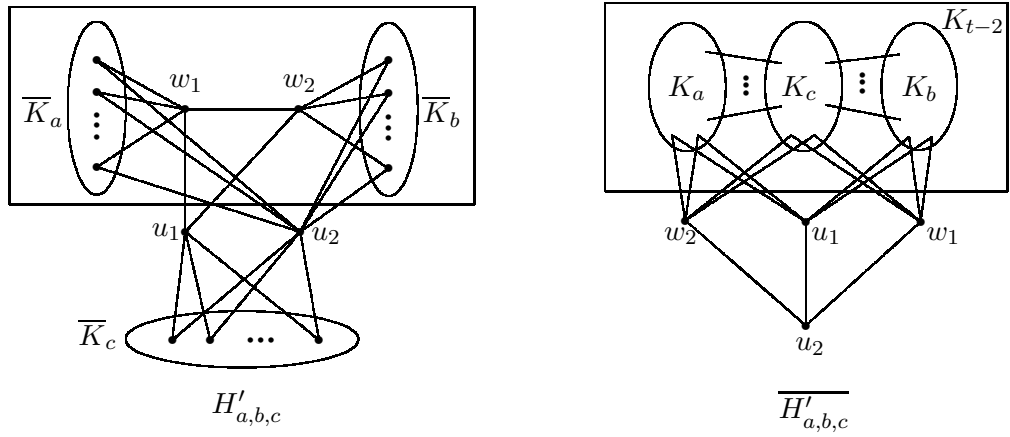
\begin{figure}[htp]
\setlength{\unitlength}{0.75pt}
\begin{center}
\begin{picture}(500.3,213.2) \qbezier(60.2,160.2)(60.2,142.2)(55.9,129.5)
\qbezier(55.9,129.5)(51.7,116.7)(45.7,116.7)\qbezier(45.7,116.7)(39.7,116.7)(35.4,129.5)
\qbezier(35.4,129.5)(31.2,142.2)(31.2,160.2)\qbezier(31.2,160.2)(31.2,178.2)(35.4,191.0)
\qbezier(35.4,191.0)(39.7,203.7)(45.7,203.7)\qbezier(45.7,203.7)(51.7,203.7)(55.9,191.0)
\qbezier(55.9,191.0)(60.2,178.2)(60.2,160.2)\qbezier(205.9,159.5)(205.9,141.5)(201.7,128.7)
\qbezier(201.7,128.7)(197.4,116.0)(191.4,116.0)\qbezier(191.4,116.0)(185.4,116.0)(181.1,128.7)
\qbezier(181.1,128.7)(176.9,141.5)(176.9,159.5)\qbezier(176.9,159.5)(176.9,177.5)(181.1,190.3)
\qbezier(181.1,190.3)(185.4,203.0)(191.4,203.0)\qbezier(191.4,203.0)(197.4,203.0)(201.7,190.3)
\qbezier(201.7,190.3)(205.9,177.5)(205.9,159.5)\put(89.9,159.5){\circle*{4}}
\put(146.5,159.5){\circle*{4}}\qbezier(89.9,159.5)(118.2,159.5)(146.5,159.5)
\put(45.7,184.2){\circle*{4}}\put(45.7,130.5){\circle*{4}}
\put(45.7,168.2){\circle*{4}}\put(191.4,184.2){\circle*{4}}
\put(191.4,129.8){\circle*{4}}\put(191.4,167.5){\circle*{4}}
\put(191.4,155.9){\circle*{2}}\put(191.4,145.0){\circle*{2}}
\put(191.4,150.1){\circle*{2}}\put(45.7,156.6){\circle*{2}}
\put(45.7,145.7){\circle*{2}}\put(45.7,151.5){\circle*{2}}
\put(89.9,97.9){\circle*{4}}\put(147.2,97.9){\circle*{4}}
\qbezier(89.9,159.5)(89.9,128.7)(89.9,97.9)\qbezier(146.5,159.5)(118.2,128.7)(89.9,97.9)
\qbezier(45.7,184.2)(67.8,171.8)(89.9,159.5)\qbezier(45.7,168.2)(67.8,163.9)(89.9,159.5)
\qbezier(89.9,159.5)(67.8,145.0)(45.7,130.5)\qbezier(146.5,159.5)(168.9,171.8)(191.4,184.2)
\qbezier(146.5,159.5)(168.9,163.5)(191.4,167.5)\qbezier(146.5,159.5)(168.9,144.6)(191.4,129.8)
\qbezier(45.7,184.2)(96.4,141.0)(147.2,97.9)\qbezier(45.7,168.2)(96.4,133.0)(147.2,97.9)
\qbezier(45.7,130.5)(96.4,114.2)(147.2,97.9)\qbezier(191.4,184.2)(169.3,141.0)(147.2,97.9)
\qbezier(191.4,167.5)(169.3,132.7)(147.2,97.9)\qbezier(191.4,129.8)(169.3,113.8)(147.2,97.9) \qbezier(176.2,43.9)(176.2,38.0)(159.2,33.9)\qbezier(159.2,33.9)(142.2,29.7)(118.2,29.7)
\qbezier(118.2,29.7)(94.2,29.7)(77.2,33.9)\qbezier(77.2,33.9)(60.2,38.0)(60.2,43.9)
\qbezier(60.2,43.9)(60.2,49.7)(77.2,53.9)\qbezier(77.2,53.9)(94.2,58.0)(118.2,58.0)
\qbezier(118.2,58.0)(142.2,58.0)(159.2,53.9)\qbezier(159.2,53.9)(176.2,49.7)(176.2,43.9)
\put(83.4,43.5){\circle*{4}}\put(156.6,43.5){\circle*{4}}
\put(102.2,43.5){\circle*{4}}\put(122.5,43.5){\circle*{2}}
\put(132.7,43.5){\circle*{2}}\put(127.6,43.5){\circle*{2}}
\qbezier(89.9,97.9)(86.6,70.7)(83.4,43.5)\qbezier(89.9,97.9)(96.1,70.7)(102.2,43.5)
\qbezier(89.9,97.9)(123.3,70.7)(156.6,43.5)\qbezier(147.2,97.9)(115.3,70.7)(83.4,43.5)
\qbezier(147.2,97.9)(124.7,70.7)(102.2,43.5)\qbezier(147.2,97.9)(151.9,70.7)(156.6,43.5)
\put(9.0,166.0){\makebox(0,0)[tl]{$\overline{K}_a$}}
\put(208.4,166.0){\makebox(0,0)[tl]{$\overline{K}_b$}}
\put(36.6,54){\makebox(0,0)[tl]{$\overline{K}_c$}}
\put(2.2,211.0){\line(1,0){232.0}\line(0,-1){99.3}}\put(234.2,111.7){\line(-1,0){232.0}\line(0,1){99.3}}
\put(103,18.0){\makebox(0,0)[tl]{$H'_{a,b,c}$}} \qbezier(359.6,170.7)(359.6,157.4)(354.0,147.9)
\qbezier(354.0,147.9)(348.3,138.5)(340.4,138.5)
\qbezier(340.4,138.5)(332.4,138.5)(326.8,147.9)\qbezier(326.8,147.9)(321.2,157.4)(321.2,170.7)
\qbezier(321.2,170.7)(321.2,184.1)(326.8,193.6)\qbezier(326.8,193.6)(332.4,203.0)(340.4,203.0)
\qbezier(340.4,203.0)(348.3,203.0)(354.0,193.6)\qbezier(354.0,193.6)(359.6,184.1)(359.6,170.7) \qbezier(417.6,171.1)(417.6,157.6)(412.0,148.0)\qbezier(412.0,148.0)(406.3,138.5)(398.4,138.5)
\qbezier(398.4,138.5)(390.4,138.5)(384.8,148.0)\qbezier(384.8,148.0)(379.2,157.6)(379.2,171.1)
\qbezier(379.2,171.1)(379.2,184.6)(384.8,194.2)\qbezier(384.8,194.2)(390.4,203.7)(398.4,203.7)
\qbezier(398.4,203.7)(406.3,203.7)(412.0,194.2)\qbezier(412.0,194.2)(417.6,184.6)(417.6,171.1)
\qbezier(474.9,171.1)(474.9,157.9)(469.4,148.5)\qbezier(469.4,148.5)(463.8,139.2)(456.0,139.2)
\qbezier(456.0,139.2)(448.2,139.2)(442.7,148.5)\qbezier(442.7,148.5)(437.2,157.9)(437.2,171.1)
\qbezier(437.2,171.1)(437.2,184.3)(442.7,193.7)\qbezier(442.7,193.7)(448.2,203.0)(456.0,203.0)
\qbezier(456.0,203.0)(463.8,203.0)(469.4,193.7)\qbezier(469.4,193.7)(474.9,184.3)(474.9,171.1)
\qbezier(349.5,188.5)(368.3,185.6)(387.2,182.7)
\qbezier(350.9,154.4)(369.0,157.7)(387.2,161.0)\qbezier(409.6,182.7)(427.8,184.9)(445.9,187.1)
\qbezier(409.6,161.7)(428.1,158.1)(446.6,154.4)\put(399.5,103.0){\circle*{4}}
\put(399.5,50.8){\circle*{4}}\qbezier(399.5,103.0)(399.5,76.9)(399.5,50.8)
\put(456.8,103.0){\circle*{4}}\qbezier(456.8,103.0)(428.1,76.9)(399.5,50.8)
\put(339.3,103.7){\circle*{4}}\qbezier(339.3,103.7)(369.4,77.2)(399.5,50.8)
\qbezier(332.8,146.5)(336.0,125.1)(339.3,103.7)\qbezier(345.8,146.5)(342.6,125.1)(339.3,103.7)
\qbezier(451.0,145.7)(453.9,124.3)(456.8,103.0)\qbezier(462.6,145.7)(459.7,124.3)(456.8,103.0)
\qbezier(332.8,146.5)(366.1,124.7)(399.5,103.0)\qbezier(462.6,145.7)(431.0,124.3)(399.5,103.0)
\qbezier(345.8,146.5)(372.7,124.7)(399.5,103.0)\qbezier(451.0,145.7)(425.2,124.3)(399.5,103.0)
\put(81.9,176.9){\makebox(0,0)[tl]{$w_1$}}\put(138.5,176.9){\makebox(0,0)[tl]{$w_2$}}
\put(73.2,102.2){\makebox(0,0)[tl]{$u_1$}}\put(152.8,100.1){\makebox(0,0)[tl]{$u_2$}}
\put(330.6,176.2){\makebox(0,0)[tl]{$K_a$}}\put(447.3,176.2){\makebox(0,0)[tl]{$K_b$}}
\put(387.2,176.2){\makebox(0,0)[tl]{$K_c$}}\put(459.7,100.2){\makebox(0,0)[tl]{$w_1$}}
\put(328.8,100.2){\makebox(0,0)[tl]{$w_2$}}\put(402.4,100.2){\makebox(0,0)[tl]{$u_1$}}
\qbezier(392.2,146.5)(365.8,125.1)(339.3,103.7)\qbezier(403.8,146.5)(371.6,125.1)(339.3,103.7)
\qbezier(392.2,146.5)(424.5,124.7)(456.8,103.0)\qbezier(403.8,146.5)(430.3,124.7)(456.8,103.0)
\put(394.4,44.5){\makebox(0,0)[tl]{$u_2$}}\put(386.0,18.0){\makebox(0,0)[tl]{$\overline{H'_{a,b,c}}$}}
\put(299.4,213.2){\line(1,0){200.8}\line(0,-1){95.0}}\put(500.3,118.2){\line(-1,0){200.8}\line(0,1){95.0}}
\put(467,210.2){\makebox(0,0)[tl]{$K_{t-2}$}}\put(427.8,176.9){\circle*{2}}
\put(427.8,168.9){\circle*{2}}\put(427.8,173.3){\circle*{2}}
\put(369.8,175.5){\circle*{2}}\put(369.8,167.5){\circle*{2}}
\put(369.8,171.8){\circle*{2}}
\end{picture}
\end{center}
\caption{The graph $H'_{a,b,c}$ and its complement.}\label{fig3.2}
\end{figure}

\begin{proof}
Now, $|W|=t$. Since $W_1\subseteq V(T_1),$ by (\ref{a3.1}) we have
$$e\left(\overline{G}\right)\geq 2|T_1|-1+2\left(|W|-|T_1|\right)=2t-1.$$
Recall that $e\left(\overline{G}\right)=2t-1.$
We can see that $e_{\overline{G}}(S)=0$, $T_1$ is a tree with $V(T_1)=W_1$,
and $W\setminus W_1$ is an independent set of $\overline{G}$.

We shall further show that $T_1$ is a star.
Since $diam(\overline{G})=2$ and $V(T_1)=W_1$, we can observe that
$w_1$ and $w_2$ have a unique and common neighbor in $S$
for any $w_1,w_2\in V(T_1)$ with $d_{T_1}(w_1,w_2)\geq 3$.
Thus, if $diam(T_1)\geq 5$, then $V(T_1)$ dominates exactly one vertex of $S$ in $\overline{G}$,
contradicting Claim \ref{cl3.5}. If $diam(T_1)=4$,
then all vertices in $V(T_1)$ but the central vertex are adjacent to a vertex (say $u_2$) of $S$.
But now, we can find a leaf $w$ with $d_{\overline{G}}(w,u_1)=3$, a contradiction.
If $diam(T_1)=3$, that is, $T_1$ is a double star,
then we can similarly see that
all leaves have a common neighbor (say $u_2$) in $S$,
and two central vertices $w_1$ and $w_2$ are adjacent to $u_1$.
It follows that $\overline{G}\cong H'_{a,b,c}$ for some $a,b,c$ with $a+b+c=t-2$ (see Fig. \ref{fig3.2}).
If $c\leq \gamma-1$, we contract the edge $u_1u_2$ in $G$ and call the new vertex $u$ in
the resulting graph, then we get a complete bipartite subgraph with bipartite partition $\left\langle V(K_c)\cup \{u\}, V(K_a)\cup V(K_b)\cup \{w_1,w_2\}\right\rangle$. This implies that $G$ contains a $K_{c+1,a+b+2}$-minor, a contradiction with $(s,t)$-property.
So, $c\geq \gamma$. By symmetry, we also have $a,b\geq \gamma.$
Therefore, $t-2=a+b+c\geq 3\gamma$.
This implies that $\gamma\leq \lfloor\frac{t-2}3\rfloor$ and thus $s\leq \lfloor\frac{t-2}3\rfloor$, which contradicts $\beta=\lfloor\frac{t+1}{s+1}\rfloor\leq2$.

Now we conclude that $T_1$ is a star.
Let $w$ be the central vertex of $T_1$.
Without loss of generality, assume that $wu_1\in E(\overline{G})$.
Then  $\overline{G}\cong H_{a,b,c}$ for some $a,b,c$ with $a+b+c=t-1$ (see Fig. \ref{fig3.1}),
as desired.
\end{proof}

\begin{claim}\label{cl3.7}
If $|S|=3$, then $\overline{G}$ is isomorphic to the Petersen graph $H^\star$ (see Fig. \ref{fig3.3}).
\end{claim}

\begin{figure}[htp]
\setlength{\unitlength}{0.75pt}
\begin{center}
\begin{picture}(448.1,161.7)
\put(30.5,146.5){\circle*{4}}\put(107.3,146.5){\circle*{4}}
\qbezier(30.5,146.5)(68.9,146.5)(107.3,146.5)
\put(0.0,88.5){\circle*{4}}\qbezier(30.5,146.5)(15.2,117.5)(0.0,88.5)
\put(29.7,29.0){\circle*{4}}
\qbezier(0.0,88.5)(14.9,58.7)(29.7,29.0)
\put(113.1,29.0){\circle*{4}}
\qbezier(29.7,29.0)(71.4,29.0)(113.1,29.0)
\put(138.5,87.0){\circle*{4}}
\qbezier(107.3,146.5)(122.9,116.7)(138.5,87.0)\qbezier(138.5,87.0)(125.8,58.0)(113.1,29.0)
\qbezier(30.5,146.5)(111.7,97.9)(113.1,29.0)\qbezier(107.3,146.5)(29.7,96.4)(29.7,29.0)
\qbezier(0.0,88.5)(73.2,47.9)(138.5,87.0)\put(52.9,97.2){\circle*{4}}
\put(71.1,68.2){\circle*{4}}\put(89.2,95.7){\circle*{4}}
\put(71.1,87.0){\circle*{4}}\qbezier(52.9,97.2)(62.0,92.1)(71.1,87.0)
\qbezier(71.1,87.0)(71.1,77.6)(71.1,68.2)\qbezier(71.1,87.0)(80.1,91.4)(89.2,95.7)
\put(237.1,126.9){\circle*{4}}\put(352.4,126.9){\circle*{4}}
\put(294.4,126.9){\circle*{4}}\put(237.8,45.7){\circle*{4}}
\put(351.6,45.7){\circle*{4}}\put(294.4,45.7){\circle*{4}}
\qbezier(237.1,128.3)(295.1,161.7)(352.4,128.3)\qbezier(237.1,45.7)(295.1,13.1)(352.4,45.7)
\put(411.1,126.9){\circle*{4}}\put(411.1,45.7){\circle*{4}}
\qbezier(411.1,126.9)(448.1,87.0)(411.1,45.7)\put(411.1,87.7){\circle*{4}}
\qbezier(411.1,126.9)(411.1,107.3)(411.1,87.7)\qbezier(411.1,87.7)(411.1,66.7)(411.1,45.7)
\qbezier(294.4,127.6)(352.4,161.7)(409.6,127.6)\qbezier(293.6,45.0)(353.8,11.6)(409.6,45.0)
\qbezier(237.1,126.9)(294.4,86.3)(351.6,45.7)\qbezier(237.1,126.9)(265.7,86.3)(294.4,45.7)
\qbezier(294.4,126.9)(266.1,86.3)(237.8,45.7)\qbezier(294.4,126.9)(294.4,86.3)(294.4,45.7)
\qbezier(294.4,126.9)(323.0,86.3)(351.6,45.7)\qbezier(352.4,126.9)(323.4,86.3)(294.4,45.7)
\qbezier(352.4,126.9)(295.1,86.3)(237.8,45.7)\qbezier(352.4,126.9)(381.7,126.9)(411.1,126.9)
\qbezier(294.4,45.7)(352.7,86.3)(411.1,126.9)\qbezier(237.8,45.7)(324.4,86.3)(411.1,126.9)
\qbezier(237.1,126.9)(324.1,107.3)(411.1,87.7)\qbezier(352.4,126.9)(381.7,107.3)(411.1,87.7)
\qbezier(411.1,87.7)(381.4,66.7)(351.6,45.7)\qbezier(411.1,87.7)(324.4,66.7)(237.8,45.7)
\qbezier(294.4,126.9)(352.7,86.3)(411.1,45.7)\qbezier(237.1,126.9)(324.1,86.3)(411.1,45.7)
\qbezier(351.6,45.7)(381.4,45.7)(411.1,45.7)\put(206.6,87.0){\circle*{4}}
\qbezier(237.1,126.9)(221.9,106.9)(206.6,87.0)\qbezier(206.6,87.0)(222.2,66.3)(237.8,45.7)
\qbezier(206.6,87.0)(250.5,106.9)(294.4,126.9)\qbezier(206.6,87.0)(250.5,66.3)(294.4,45.7)
\qbezier(206.6,87.0)(279.5,106.9)(352.4,126.9)\qbezier(206.6,87.0)(279.1,66.3)(351.6,45.7)
\put(319.0,10.0){\makebox(0,0)[tl]{$\overline{H^\star}$}}\put(67.4,10.0){\makebox(0,0)[tl]{$H^\star$}}
\put(34.9,101.9){\makebox(0,0)[tl]{$u_1$}}\put(92.9,101.9){\makebox(0,0)[tl]{$u_2$}}
\put(63.9,62.9){\makebox(0,0)[tl]{$u_3$}}
\end{picture}
\caption{The Petersen graph $H^\star$ and its complement.}\label{fig3.3}
\end{center}
\end{figure}
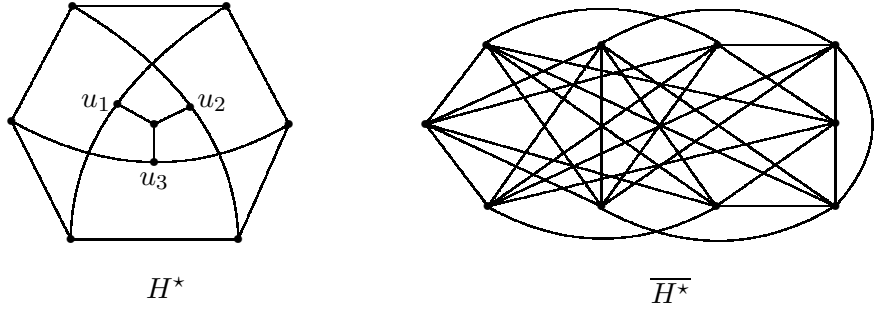

\begin{proof}
Now $\delta\left(\overline{G}\right)\geq |S|=3$. So, all leaves of $T_1$, if exist, belong to $W\setminus W_1$.
Hence,
\begin{align}\label{a3.00}
e_{\overline{G}}(V(T_1))+e_{\overline{G}}(V(T_1),S)\geq 2|T_1|,
\end{align}
and if equality holds, then $T_1$
is a cycle with $V(T_1)=W_1$.

On the other hand, by Claim \ref{cl3.5}, $V(T_1)$ dominates $S$ in $\overline{G}$, which implies that
each vertex of $W\setminus V(T_1)$ also dominates $S$ in $\overline{G}$
(Otherwise, we can find $w_1\in V(T_1)$ and $w_2\in W\setminus V(T_1)$ with $d_{\overline{G}}(w_1,w_2)\geq3$).
Therefore,
$e_{\overline{G}}\left(W\setminus V(T_1),S\right)=3(|W|-|T_1|)$.
Combining with (\ref{a3.1}) and (\ref{a3.00}), we have
$$e\left(\overline{G}\right)\geq e_{\overline{G}}(S)+2|T_1|+3(|W|-|T_1|)\geq3|W|-|T_1|\geq 2|W|+1=2t-1.$$
Since $e\left(\overline{G}\right)=2t-1$, we have
$e_{\overline{G}}(S)=0$, $T_1\cong C_{|W_1|}$ and $|W|=|T_1|+1.$

Note that $\delta\left(\overline{G}\right)\geq 3$.
Each vertex of $S$ has at least two neighbors in $V(T_1)$.
It follows that $|T_1|\geq6$.
Furthermore,
any two vertices with distance at least 3 in $T_1$ have a common neighbor in $S$.
If $|T_1|\geq 7$, then all vertices in $T_1$ have a common neighbor in $S$, contradicting Claim \ref{cl3.5}.
So, $T_1\cong C_6$ and every pair of vertices with distance 3 in $T_1$ have a common neighbor in $S$.
Thus, $\overline{G}\cong H^\star$ (see Fig. \ref{fig3.3}).
\end{proof}

Combining with Claim \ref{cl3.6} and Claim \ref{cl3.7},
the proof of Lemma \ref{lem3.3} is completed.
\end{proof}

We now use $\mathbb{H}_i$, $\mathbb{H}_{>i}$ and $\mathbb{H}_{<i}$
to denote the family of components in $G^\star-K$ with order $i$,
greater than $i$ and less than $i$, respectively.

\begin{lem}\label{lem3.4}
$\mathbb{H}_{>t+3}=\varnothing.$
\end{lem}

\begin{proof}
Suppose to the contrary that $H\in \mathbb{H}_{>t+3}$. By Lemma \ref{lem2.2},
\begin{align}\label{a3.3}
e(H)\leq {t\choose2}+|H|-t.
\end{align}
Assume that $|H|=pt+q$, where $p\geq 1$ and $1\leq q\leq t$,
and let $H'\cong pK_t\cup K_q$ with $V(H')=V(H)$.
Clearly, $H'$ satisfies $(s,t)$-property.
By (\ref{a3.3}), we can see that
$$e(H)\leq{t\choose2}+(pt+q)-t<p{t\choose 2}+{q\choose 2}=e(H')$$
for $|H|>t+3$ and $t\geq 4$.
If $p\leq7$, then $|H|\leq 8t$ (a constant). By Lemma \ref{lem2.4},
we have $e(H')\leq e(H)$, a contradiction.
Now assume that $p\geq8.$ Then by (\ref{a3.3}),
\begin{align}\label{a3.4}
e(H)\leq {t\choose2}+pt<\frac {4p}5{t\choose 2}\leq\frac 45e(H').
\end{align}
Now let $G'=G^\star-E(H)+E(H')$ and $\rho'=\rho(G')$. Then
\begin{eqnarray*}
\rho'-\rho\geq X^T(A(G')-A(G^\star))X
=\sum\limits_{uv\in E(H')}2x_ux_v-\sum\limits_{uv\in E(H)}2x_ux_v.
\end{eqnarray*}
Recall that $x_1=\max_{v\in V(H)}x_v$ and $x_2=\min_{v\in V(H)}x_v.$
By (\ref{a3.4}) and Claim \ref{cl2.1},
$$\rho'-\rho\geq 2e(H')x^2_2-2e(H)x^2_1>2e(H')(x^2_2-\frac 45x^2_1)>0,$$
a contradiction.
Thus we have $\mathbb{H}_{>t+3}=\varnothing.$
\end{proof}

\begin{lem}\label{lem3.5}
$\mathbb{H}_t=O(\frac nt)$.
\end{lem}

\begin{proof}
By Lemma \ref{lem3.4}, $\mathbb{H}_{>t+3}=\varnothing.$
Hence, it suffices to show $|\mathbb{H}_i|<t$
for any $i\leq t+3$ and $i\neq t$.
Suppose to the contrary that $|\mathbb{H}_i|\geq t$ for some $i$
and let $\mathbb{F}$ be the disjoint union of any $t$ components in $\mathbb{H}_i$.
If $i<t$, then $e(\mathbb{F})\leq e(tK_i)<e(iK_t)$, which contradicts Lemma \ref{lem2.4}.

If $i=t+1$, then by Lemma \ref{lem3.1} and $\beta=\lfloor\frac{t+1}{s+1}\rfloor\leq\lfloor\frac{t+1}{3}\rfloor$,
we have
$$e(\mathbb{F})\leq t\left({t\choose 2}+\beta-1\right)<(t+1){t\choose 2}=e\left((t+1)K_t\right),$$
which contradicts Lemma \ref{lem2.4}.

If $i\in\{t+2,t+3\}$, then by Lemma \ref{lem2.2},
$$e(\mathbb{F})\leq t\left({t\choose 2}+i-t\right)<i{t\choose 2}=e(iK_t),$$
also a contradiction.
\end{proof}

Now we are ready to characterize a $(t+2)$-vertex component of $G^\star-K$, which
is also a key subgraph of the extremal graphs.

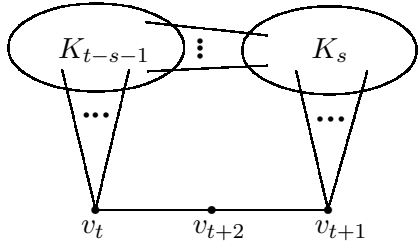
\begin{figure}[htp]
 \setlength{\unitlength}{0.75pt}
\begin{center}
 \begin{picture}(203.7,140.7)
 \qbezier(87.7,118.5)(87.7,109.4)(74.9,102.9)\qbezier(74.9,102.9)(62.0,96.4)(43.9,96.4)
 \qbezier(43.9,96.4)(25.7,96.4)(12.8,102.9)\qbezier(12.8,102.9)(0.0,109.4)(0.0,118.5)
 \qbezier(0.0,118.5)(-0.0,127.7)(12.8,134.2)\qbezier(12.8,134.2)(25.7,140.7)(43.9,140.7)
 \qbezier(43.9,140.7)(62.0,140.7)(74.9,134.2)\qbezier(74.9,134.2)(87.7,127.7)(87.7,118.5)
 \qbezier(203.7,117.1)(203.7,107.9)(191.0,101.5)\qbezier(191.0,101.5)(178.2,95.0)(160.2,95.0)
 \qbezier(160.2,95.0)(142.2,95.0)(129.5,101.5)\qbezier(129.5,101.5)(116.7,107.9)(116.7,117.1)
 \qbezier(116.7,117.1)(116.7,126.2)(129.5,132.7)\qbezier(129.5,132.7)(142.2,139.2)(160.2,139.2)
 \qbezier(160.2,139.2)(178.2,139.2)(191.0,132.7)\qbezier(191.0,132.7)(203.7,126.2)(203.7,117.1)
 \put(43.5,37.0){\circle*{4}}
 \put(159.5,37.0){\circle*{4}}
 \qbezier(43.5,37.0)(101.5,37.0)(159.5,37.0)
 \put(101.5,37.0){\circle*{4}}
 \put(39.0,84.8){\circle*{2}}
 \put(48.8,84.8){\circle*{2}}
 \put(43.5,84.8){\circle*{2}}
 \put(155.7,82.7){\circle*{2}}
 \put(166.2,82.7){\circle*{2}}
 \put(161.0,82.7){\circle*{2}}
 \put(25,123.3){\makebox(0,0)[tl]{$K_{t-s-1}$}}
 \put(151.4,123.3){\makebox(0,0)[tl]{$K_s$}}
 \qbezier(68.9,131.2)(99.0,128.0)(129.1,124.7)
 \qbezier(69.6,106.6)(99.3,108.0)(129.1,109.5)
 \qbezier(26.8,107.3)(35.2,72.1)(43.5,37.0)
 \qbezier(60.2,107.3)(51.8,72.1)(43.5,37.0)
 \qbezier(179.1,106.6)(169.3,71.8)(159.5,37.0)
 \qbezier(143.6,106.6)(151.5,71.8)(159.5,37.0)
 \put(95.7,122.0){\circle*{2}}
 \put(95.7,114.4){\circle*{2}}
 \put(95.7,118.2){\circle*{2}}
 \put(36.3,31.5){\makebox(0,0)[tl]{$v_t$}}
 \put(92.1,31.5){\makebox(0,0)[tl]{$v_{t+2}$}}
 \put(153.0,31.5){\makebox(0,0)[tl]{$v_{t+1}$}}
 \end{picture}
 \caption{The graph $S^1\left(\overline{H_{s,t}}\right)$ for
 $\beta=\lfloor\frac{t+1}{s+1}\rfloor=2$.}\label{fig3.4}
\end{center}
\end{figure}

\begin{thm}\label{thm3.2}
Let $H$ be a component of $G^\star-K$. If $|H|=t+2$, then $\beta\leq2$. Moreover,
$H\cong S^1\left(\overline{H_{s,t}}\right)$ for $\beta=2$ (see Fig. \ref{fig3.4}),
and $H\cong \overline{H^\star}$ for $\beta=1$.
\end{thm}

\begin{proof}
The proof is divided into several claims.

\begin{claim}\label{cl3.8}
$\beta\leq2$.
\end{claim}

\begin{proof} If $\beta\geq3$, then by Theorem \ref{thm3.1},
$e\left(\overline{H_{s,t}}\right)={t\choose 2}+\beta-1\geq {t\choose 2}+2.$
By Lemma \ref{lem3.5}, $G^\star-K$ contains a component $H_0$ isomorphic to $K_t$.
Let $H'\cong\overline{H_{s,t}}\cup \overline{H_{s,t}}$ and $V(H')=V(H\cup H_0)$.
By Theorem \ref{thm3.1}, $H'$ has $(s,t)$-property. However,
$e(H\cup H_0)=2{t\choose 2}+2<e(H'),$ contradicting Lemma \ref{lem2.4}.
So, $\beta\leq2$.
\end{proof}

\begin{claim}\label{cl3.9}
$e(H)={t\choose 2}+2$.
\end{claim}

\begin{proof}
By Lemma \ref{lem2.2}, we have $e(H)\leq {t\choose 2}+2$.
Since $K_t\cup K_2$ satisfies $(s,t)$-property,
by Lemma \ref{lem2.4}, we have $e(H)\geq e(K_t\cup K_2)={t\choose 2}+1$.
Now suppose that $e(H)=e(K_t\cup K_2)$.
Since $\Delta(H)\leq t-1$ and $\delta(H)\geq1$,
we can see that $\pi(H)\prec \pi(K_t\cup K_2)$.
By Lemma \ref{lem2.8}, we have $$\pi(H)=\pi(K_t\cup K_2)=(t-1,t-1,\ldots,t-1,1,1),$$
which implies that $H$ is obtained from $K_t$ by deleting an edge $u_1u_2$ and adding
two pendant edges $u_1v_1$ and $u_2v_2$.
Since $t\geq4$, we have $N_H(u_1)\setminus\{v_1\}\neq\varnothing$ and thus
$$\rho (x_{u_1}-x_{v_1})=(x_{v_1}-x_{u_1})+\sum_{v\in N_H(u_1)\setminus\{v_1\}}x_v>x_{v_1}-x_{u_1},$$
that is, $x_{u_1}>x_{v_1}$. By symmetry, we have $x_{u_1}=x_{u_2}$ and $x_{v_1}=x_{v_2}$.
Now, let $H'=H-\{u_1v_1, u_2v_2\}+\{u_1u_2, v_1v_2\}$ and $G'=G^\star-E(H)+E(H')$.
Then $H'\cong K_t\cup K_2$ and thus $H'$ has $(s,t)$-property.
\begin{eqnarray*}
\rho(G')-\rho(G^\star)\geq X^T(A(G')-A(G^\star))X
=2(x_{u_1}-x_{v_2})(x_{u_2}-x_{v_1})>0,
\end{eqnarray*}
a contradiction. Hence, $e(H)={t\choose 2}+2$.
\end{proof}

Now, combining Claims \ref{cl3.8} and \ref{cl3.9} with Lemma \ref{lem3.3},
$H$ is isomorphic to either some $\overline{H_{a,b,c}}$ (see Fig. \ref{fig3.1})
or the complement of the Peterson graph $H^\star$.

\begin{claim}\label{cl3.10}
If $\beta=2$, then $H\cong S^1\left(\overline{H_{s,t}}\right)$ (see Fig. \ref{fig3.4}).
\end{claim}

\begin{proof}
Since $\beta=\lfloor\frac{t+1}{s+1}\rfloor$, we have
\begin{align}\label{a3.5}
2s+1\leq t\leq 3s+1~~~\hbox{and}~~~\gamma=\min\{s,\lfloor\frac{t+1}2\rfloor\}=s.
\end{align}
If $H\cong \overline{H^\star}$, then $|H|=10$, $t=8$ and $H$ is 6-regular.
It follows that $e(H)=30.$
On the other hand, since $S^1\left(\overline{H_{s,8}}\right)$ is a subdivision of $\overline{H_{s,8}}$,
by Theorem \ref{thm3.1} we have
$e\left(S^1\left(\overline{H_{s,8}}\right)\right)=e\left(\overline{H_{s,8}}\right)+1={t\choose 2}+\beta=30$.
By Fig. \ref{fig3.4} we can see that
\begin{align}\label{a3.6}
\pi\left(S^1\left(\overline{H_{s,t}}\right)\right)=(t-1,\ldots,t-1,t-s,s+1,2),
\end{align}
where $s+1\leq t-s\leq6$. Since $H$ is 6-regular, we have
$\pi(H)\prec\pi\left(S^1\left(\overline{H_{s,8}}\right)\right),$
contradicting Lemma \ref{lem2.8}.
Thus, $H$ is isomorphic to some $\overline{H_{a,b,c}}$ with $a+b+c=t-1$.

Next we show that
\begin{align}\label{a3.7}
\min\{b,c\}\geq \gamma.
\end{align}
If $b\leq \gamma-1$, we contract the edge $u_2w$ in $\overline{H_{a,b,c}}$ and call the new vertex $u$ in
the resulting graph, then we get a complete bipartite subgraph with bipartite partition
$\langle V(K_b)\cup \{u\}$, $V(K_a)\cup V(K_c)\cup \{u_1\}\rangle$.
This implies that $H$ contains a $K_{b+1,a+c+1}$-minor, a contradiction with $(s,t)$-property.
So, $b\geq \gamma$. And by symmetry, $c\geq \gamma.$

Now by Fig. \ref{fig3.1} we can see that
\begin{align}\label{a3.8}
\pi\left(\overline{H_{a,b,c}}\right)=(t-1,\ldots,t-1,a_1,a_2,a_3),
\end{align}
where $a_1,a_2,a_3\in\{a+2,b+1,c+1\}.$
By (\ref{a3.5}) and (\ref{a3.7}), $\min\{b,c\}\geq \gamma=s.$
It follows that $a_3\geq2$ and $a_2\geq s+1$.
Comparing (\ref{a3.6}) with (\ref{a3.8}), we have
$\pi\left(\overline{H_{a,b,c}}\right)\prec\pi\left(S^1\left(\overline{H_{s,t}}\right)\right)$.
By Lemma \ref{lem2.8}, $\pi\left(\overline{H_{a,b,c}}\right)=\pi\left(S^1\left(\overline{H_{s,t}}\right)\right)$.
So, $a_3=2$ and $a_2=s+1.$ Note that $\min\{b,c\}\geq s\geq2$.
We conclude that $a=0$ and $\min\{b,c\}=s$, that is, $\overline{H_{a,b,c}}\cong S^1\left(\overline{H_{s,t}}\right).$

Finally, we shall prove that $S^1\left(\overline{H_{s,t}}\right)$ has $(s,t)$-property.
It suffices to show that contracting any edge, the resulting graph always
has $(s,t)$-property.
Indeed, if we contract an edge within the $(t-1)$-clique, then the complement of the resulting graph
is clearly connected. By Lemma \ref{lem3.0}, it has $(s,t)$-property.
If we contract an edge out of the $(t-1)$-clique, then the resulting graph
is isomorphic to $\overline{H_{s,t}}$.
By Theorem \ref{thm3.1}, it also has $(s,t)$-property.
\end{proof}

\begin{claim}\label{cl3.11}
If $\beta=1$, then $H\cong \overline{H^\star}$ (see Fig. \ref{fig3.3}).
\end{claim}

\begin{proof}
Since $\beta=\lfloor\frac{t+1}{s+1}\rfloor=1$, we have $t\leq 2s$
and thus $\gamma=\min\{s,\lfloor\frac{t+1}{2}\rfloor\}=\lfloor\frac{t+1}{2}\rfloor.$
If $H$ is isomorphic to some $\overline{H_{a,b,c}}$ with $a+b+c=t-1$,
then by (\ref{a3.7}), we have $t-1\geq b+c\geq 2\gamma=2\lfloor\frac{t+1}{2}\rfloor,$
a contradiction.
By Lemma \ref{lem3.3}, $H$ is only possibly isomorphic to $\overline{H^\star}.$

It remains to show that $\overline{H^\star}$ has $(s,t)$-property.
We know that the Peterson graph $H^\star$ is 3-connected
and any two non-adjacent vertices of $H^\star$ have exactly one common neighbor, which implies that
contracting any edge of $\overline{H^\star}$,
the complement of the resulting graph is connected.
By Lemma \ref{lem3.0}, the resulting graph has $(s,t)$-property,
and thus $\overline{H^\star}$ has $(s,t)$-property.
\end{proof}

Combining with Claims \ref{cl3.8}, \ref{cl3.10} and \ref{cl3.11},
the proof of Theorem \ref{thm3.2} is completed.
\end{proof}

From Lemma \ref{lem3.4}, we know that $\mathbb{H}_{>t+3}=\varnothing$.
Now combining with Lemma \ref{lem3.3}, we can get a stronger result.

\begin{thm}\label{thm3.3}
$\mathbb{H}_{>t+2}=\varnothing$.
\end{thm}

\begin{proof}
It suffices to show that $\mathbb{H}_{t+3}=\varnothing$.
Suppose to the contrary that there exists a component $H$ of $G^\star-K$ with $|H|=t+3.$
On one hand, $e(H)\leq{t\choose 2}+3$ by Lemma \ref{lem2.2}.
On the other hand, note that
$K_t\cup K_3$ also satisfies $(s,t)$-property.
Then by Lemma \ref{lem2.4}, $e(H)\geq e(K_t\cup K_3)={t\choose 2}+3.$
Therefore, $e(H)={t\choose 2}+3$.
Moreover, if $\beta\geq3$, then by Theorem \ref{thm3.1},
$|\overline{H_{s,t}}|=t+1$, $e\left(\overline{H_{s,t}}\right)={t\choose 2}+(\beta-1)$
and $\overline{H_{s,t}}$ has $(s,t)$-property.
Selecting two copies of $K_t$ in $G^\star-K$, we have
$$e(H\cup K_t\cup K_t)=3{t\choose 2}+3<3{t\choose 2}+3(\beta-1)
=e\left(\overline{H_{s,t}}\cup \overline{H_{s,t}}\cup \overline{H_{s,t}}\right),$$
a contradiction with Lemma \ref{lem2.4}. Therefore, $\beta\leq2.$

\begin{claim}\label{cl3.12}
$\delta(H)\geq2.$
\end{claim}

\begin{proof}
If $\delta(H)=1$, say $d_H(v)=1$,
then $H-\{v\}$ is connected and $e(H-\{v\})={t\choose 2}+2$.
By Lemma \ref{lem3.3}, $H-\{v\}$ is isomorphic to either $\overline{H^\star}$ or some $\overline{H_{a,b,c}}$ with $a+b+c=t-1$.
Observe that $\overline{H_{a,b,c}}$ contains a $K_t$-minor (by contracting edges $u_1u_2$ and $u_2w$, see Fig. \ref{fig3.1}).
Hence, $H-\{v\}\cong \overline{H^\star}$ (Otherwise, $H$ contains a $K_{1,t}$-minor).

Now let $vuw$ be a path of length 2 in $H$.
We know that any two non-adjacent vertices in the Petersen graph $H^\star$ have exactly one common neighbor.
So, contracting $uw$ in $H-\{v\}$, the new vertex is of degree $t-1$ in
the resulting graph. Correspondingly,
contracting $uw$ in $H$, the new vertex is of degree $t$,
which contradicts that $H$ is $K_{1,t}$-minor free.
\end{proof}

Now we have $\pi(H)\prec\pi(K_t\cup K_3)=(t-1,\ldots,t-1,2,2,2)$,
since $e(H)={t\choose 2}+3$ and $\delta(H)\geq2$.
By Lemma \ref{lem2.8}, $\pi(H)=\pi(K_t\cup K_3).$
Let $S_1=\{v\in V(H)|d_H(v)=2\}$ and $S_2=\cup_{v\in S_1}N_H(v)\setminus S_1$.
Then $|S_1|=3$, and $1\leq d_{S_1}(u)\leq 3$ for any $u\in S_2$.

\begin{claim}\label{cl3.13}
$d_{S_1}(u)=1$ for any $u\in S_2$.
\end{claim}

\begin{proof}
Let $R$ be the set of non-adjacent vertex-pairs in $S_2$.
Since $|V(H)\setminus S_1|=t$ and $d_H(u)=t-1$ for any $u\in V(H)\setminus S_1$,
we have $|R|=\frac12e_H(S_1,S_2)\leq3$.
Suppose that $d_{S_1}(u_0)=c\in \{2,3\}$ for some $u_0\in S_2$.
Then, $u_0$ has exactly $c$ non-neighbors, say $\{u_1,\ldots,u_c\}$, in $S_2$.
If $c=3$, then $R=\{(u_0,u_i)|i=1,2,3\}$ and there are three paths $u_0v_iu_i$ in $H$,
where $v_i\in S_1$ and $i\in\{1,2,3\}$.
But now we find a $K_{2,t-1}$-minor in $H$ by contracting two of the three paths into edges.
Therefore, $c=2$.

Since $t\geq4$, we can find a vertex $u_3\in N_H(u_0)\setminus S_1$.
If both $u_1$ and $u_2$ are neighbors of $u_3$,
then $H$ contains a double star with a non-pendant edge $u_0u_3$ and $t$ leaves,
and thus a $K_{1,t}$-minor.
If both $u_1$ and $u_2$ are not neighbors of $u_3$, then $|R|\geq4$,
a contradiction.
Hence, we may assume that $u_1u_3\in E(H)$ and $u_2u_3\notin E(H)$.
Then, $u_1u_2\in E(H)$ (Otherwise, we get $|R|\geq4$ again).
It follows that $P=u_0u_3u_1u_2$ is an induced path in $H$
and $S_2=\{u_0,u_1,u_2,u_3\}$.
Furthermore, $d_{S_1}(u_0)=d_{S_1}(u_2)=2$ and $d_{S_1}(u_1)=d_{S_1}(u_3)=1$.
Now, we can always find a double star with
a non-pendant edge in $E(P)$ and $t$ leaves, a contradiction.
\end{proof}

\begin{claim}\label{cl3.14}
$x_u>x_v$ for any $u\in S_2$ and $v\in S_1$.
\end{claim}

\begin{proof}
By (\ref{a3}), $\rho x_{u}\geq X_0+d_H(u)x_2$ and $\rho x_{v}\leq X_0+d_H(v)x_1$.
Combining with Claim \ref{cl2.1}, we have
$\rho(x_{u}-x_{v})\geq d_H(u)x_2-d_H(v)x_1>0.$
\end{proof}

Now by Claim \ref{cl3.13}, each $u_i\in S_2$ has a unique neighbor $v_i$ in $S_1$,
and thus a unique non-neighbor $u_j$ in $S_2$.
If $v_i=v_j$ or $v_iv_j\in E(H)$ for some $(u_i,u_j)\in R$,
then $u_i$ and $u_j$ have $t-1$ common neighbors after
contracting the path $u_iv_iv_ju_j$ into
$u_iv_iu_j$ in $H$. This implies that $H$ contains a $K_{2,t-1}$-minor.
Thus, $v_i\neq v_j$ and $v_iv_j\notin E(H)$ for any $(u_i,u_j)\in R$.
Let $$H'=H-\{u_iv_i,u_jv_j|(u_i,u_j)\in R\}+\{u_iu_j,v_iv_j|(u_i,u_j)\in R\}$$ and
$G'=G^\star-E(H)+E(H')$. Note that $H'\cong K_t\cup K_3$. Thus, $H'$ has $(s,t)$-property.
Moreover, by Claim \ref{cl3.14}, we have
$$\rho(G')-\rho(G^\star)\geq 2\sum_{(u_i,u_j)\in R}(x_{u_i}x_{u_j}+x_{v_i}x_{v_j}-x_{u_i}x_{v_i}-x_{u_j}x_{v_j})=2(x_{u_i}-x_{v_j})(x_{u_j}-x_{v_i})>0,$$
a contradiction.
So, $\mathbb{H}_{t+3}=\varnothing$. This completes the proof.
\end{proof}

\section{Proof of Theorem \ref{thm3}}

~~~~Tait's conjecture has been confirmed for $s+t\leq6$.
This implies that Tait's conjecture holds for $t\leq3$.
In this section, we only need prove Theorem \ref{thm3} under that $t\geq4$.

\begin{lem}\label{lem4.1}
$|\mathbb{H}_{<t}\cup \mathbb{H}_{t+2}|\leq 1.$
\end{lem}

\begin{proof}
By the way of contradiction assume that $H_1,H_2\in \mathbb{H}_{<t}\cup \mathbb{H}_{t+2}$, and let $|H_1|+|H_2|=pt+q$,
where $0\leq p\leq2$ and $1\leq q\leq t$. For each $H_i$, if $|H_i|\leq t-1$, then $e(H_i)\leq {|H_i|\choose 2}$;
and if $|H_i|=t+2$, then $e(H_i)\leq{t\choose 2}+2$ by Lemma \ref{lem3.1}.
In any case, one can check that
$e(H_1\cup H_2)<p{t\choose 2}+{q\choose 2}=e(pK_t\cup K_q)$, a contradiction with Lemma \ref{lem2.4}.
\end{proof}

\begin{thm}\label{thm4.1}
Let $G^\star,s,t,n,p,q,\beta$ be defined as in Theorem \ref{thm3}.
If $\beta=1$, then
$$ G^\star\cong \left\{
\begin{aligned}
   &K_{s-1}\nabla \left((p-1)K_t \cup \overline{H^\star}\right) &&\hbox{for $q=2$ and $t=8$}; \\
   &K_{s-1}\nabla \left(pK_t \cup K_q\right) &&\hbox{otherwise}.
\end{aligned}
\right.
$$
\end{thm}

\begin{proof}
Since $\beta=1$, by Theorem \ref{thm3.1} we have $\mathbb{H}_{t+1}=\varnothing$,
and by Theorem \ref{thm3.3}, $\mathbb{H}_{>t+2}=\varnothing$.
Furthermore, by Lemma \ref{lem4.1}, all but at most one component $H$ of $G^\star-K$
are isomorphic to $K_t$.
Note that $|G^\star-K|=pt+q$, where $1\leq q\leq t$.
Then, either $|H|=q$ or $|H|=t+q=t+2$.

If $q\neq2$, then $H\cong K_q$ and $G^\star-K\cong pK_t \cup K_q,$ as desired.
Now assume that $q=2$. Then either $H\cong K_2$, or $H\cong \overline{H^\star}$
by Theorem \ref{thm3.2} (In this case, $t=8$).
So, if $t\neq8$, then $G^\star-K\cong pK_t\cup K_2.$
If $t=8$, then $H\cong \overline{H^\star}$
(Otherwise, $H\cong K_2$.
Then $e(K_8\cup K_2)=29<30=e\left(\overline{H^\star}\right)$,
a contradiction with Lemma \ref{lem2.4}).
\end{proof}

\begin{lem}\label{lem4.2}
$|\mathbb{H}_{t+1}|\leq 2\beta-2$, where $\beta=\lfloor\frac{t+1}{s+1}\rfloor$.
\end{lem}

\begin{proof}
The case $\mathbb{H}_{t+1}=\varnothing$ is trivial.
Assume that $\mathbb{H}_{t+1}\neq\varnothing$.
Then by Theorem \ref{thm3.1},
$\beta\geq2$, $H\cong \overline{H_{s,t}}$ and $e(H)={t\choose 2}+\beta-1$ for any $H\in \mathbb{H}_{t+1}$.

If $|\mathbb{H}_{t+1}|\geq 2\beta$,
we select $2\beta$ copies of $\overline{H_{s,t}}$ and denote it by $\mathbb{F}$.
Then $e(\mathbb{F})=2\beta\left({t\choose 2}+\beta-1\right)$.
Now let $\mathbb{F}'=2\beta K_t\cup K_{2\beta}$.
Note that $2\beta<t.$ Then $\mathbb{F}'$ satisfies $(s,t)$-property.
However,
$$e(\mathbb{F}')=2\beta{t\choose 2}+{2\beta\choose 2}>2\beta\left({t\choose 2}+\beta-1\right)=e(\mathbb{F}),$$
a contradiction. So, $|\mathbb{H}_{t+1}|\leq 2\beta-1$.

Now assume that $|\mathbb{H}_{t+1}|=2\beta-1$.
Let $\mathbb{F}$ be the disjoint union of $2\beta-1$ copies of $\overline{H_{s,t}}$,
and $\mathbb{F}'=(2\beta-1) K_t\cup K_{2\beta-1}$.
Then $e(\mathbb{F}')=e(\mathbb{F}).$
Recall that $H_{s,t}\cong K_{1,\alpha}\cup (\beta-1)K_{1,s}$, where $\alpha=t-(s+1)(\beta-1)\geq s$.
So,
\begin{align}\label{a4.1}
\delta\left(\overline{H_{s,t}}\right)=(s+1)(\beta-1)>2(\beta-1)
\end{align}
and thus $\delta(\mathbb{F})>2\beta-2.$
Note that $\pi(\mathbb{F}')=(t-1,\ldots,t-1,2\beta-2,\ldots,2\beta-2).$
We can observe that $\pi(\mathbb{F})\prec\pi(\mathbb{F}'),$ a contradiction.
Therefore, $|\mathbb{H}_{t+1}|\leq 2\beta-2$.
\end{proof}

\begin{lem}\label{lem4.3}
If $\mathbb{H}_{t+1}\neq \varnothing$, then $\mathbb{H}_{t+2}\cup\mathbb{H}_{<t}=\varnothing$.
\end{lem}

\begin{proof}
Suppose to the contrary that $\mathbb{H}_{t+2}\cup\mathbb{H}_{<t}\neq \varnothing$.
By Lemma \ref{lem4.1}, $G^\star-K$ contains a unique component $H_1\in \mathbb{H}_{t+2}\cup\mathbb{H}_{<t}.$
Now let $H_2\in \mathbb{H}_{t+1}$. Then by Theorem \ref{thm3.1},
$\beta\geq2$ and $H_2\cong \overline{H_{s,t}}$.

If $|H_1|=t+2$, then $\beta=2$ and $H_1\cong S^1\left(\overline{H_{s,t}}\right)$ by Theorem \ref{thm3.2}.
Note that $H_1$ is a subdivision of $H_2$ and $e(H_2)={t\choose 2}+\beta-1$. We have
$$e(H_1\cup H_2)=2e(H_2)+1=2\left({t\choose 2}+\beta-1\right)+1=2{t\choose 2}+3.$$
Now let $H'=K_t\cup K_t\cup K_3$.
Then $|H'|=|H_1\cup H_2|$ and $e(H')=e(H_1\cup H_2).$
By (\ref{a4.1}), $\delta(H_2)>2.$
Since $H_1$ is a subdivision of $H_2$, $\delta(H_1)=2$ and its vertex of degree two is unique.
This implies that $\pi(H_1\cup H_2)\prec \pi(H')$ and $\pi(H_1\cup H_2)\neq\pi(H')$,
a contradiction.

If $|H_1|<t$, then $H_1\cong K_{|H_1|}$. Then
\begin{align}\label{a4.2}
|H_1|\leq\beta-1.
\end{align}
Otherwise, $|H_1|>\beta-1,$
then $$e(H_1\cup H_2)={|H_1|\choose 2}+{t\choose 2}+\beta-1<{t\choose 2}+{|H_1|\choose 2}+|H_1|=e(K_t\cup K_{|H_1|+1}),$$
a contradiction with Lemma \ref{lem2.4}. On the other hand, we have
\begin{align}\label{a4.3}
|H_1|(\beta-1)\leq{|H_1|\choose 2}.
\end{align}
Indeed, recall that $|\mathbb{H}_t|=O(\frac nt)$,
thus $G^\star-K$ contains a disjoint union of $H_1$ and
$|H_1|$ copies of $K_{t}$. We denote it by $\mathbb{F}$. Then
$e(\mathbb{F})={|H_1|\choose 2}+|H_1|{t\choose 2}.$
Now let $\mathbb{F}'=|H_1|H_2$. Clearly, $|\mathbb{F}|=|\mathbb{F}'|$
and $e(\mathbb{F}')=|H_1|\left({t\choose 2}+\beta-1\right)$.
By Lemma \ref{lem2.4}, $e(\mathbb{F}')\leq e(\mathbb{F})$.
It follows that (\ref{a4.3}) holds.
However, (\ref{a4.2}) and (\ref{a4.3}) contradict each other.
\end{proof}

\begin{thm}\label{thm4.2}
Let $G^\star,s,t,n,p,q,\beta$ be defined as in Theorem \ref{thm3}.
If $\beta\geq2$, then
$$ G^\star\cong \left\{
\begin{aligned}
   &K_{s-1}\nabla \left((p-1)K_t \cup S^1\left(\overline{H_{s,t}}\right)\right) &&\hbox{if $q=\beta=2$}; \\
   &K_{s-1}\nabla \left((p-q)K_t \cup q\overline{H_{s,t}}\right) &&\hbox{if $q\leq 2(\beta-1)$ except $q=\beta=2$}; \\
   &K_{s-1}\nabla \left(pK_t \cup K_q\right) &&\hbox{if $q>2(\beta-1)$}.
\end{aligned}
\right.
$$
\end{thm}

\begin{proof}
Recall that $|G^\star-K|=pt+q$, where $1\leq q\leq t$,
and $H\cong \overline{H_{s,t}}$ for any $H\in \mathbb{H}_{t+1}$.
Moreover, we assert that if $q\leq 2(\beta-1)$, then $\mathbb{H}_{<t}=\varnothing$.
Indeed, if $G^\star-K$ contains a component $H_1$ with $|H_1|<t,$
then $\mathbb{H}_{>t}=\varnothing$ and $|\mathbb{H}_{<t}|=1$ by Lemmas \ref{lem4.1} and \ref{lem4.3}.
This implies that $|H_1|=q$.
Now by (\ref{a4.3}), $|H_1|(\beta-1)\leq{|H_1|\choose 2}$, that is, $|H_1|\geq 2(\beta-1)+1,$
contradicts $q\leq 2(\beta-1)$.

Now we distinguish two cases.
We first assume that $q\neq2$.
Then by Lemmas \ref{lem4.1} and \ref{lem4.3}, $G^\star-K$ is isomorphic to either $pK_t\cup K_q$
or $(p-q)K_t\cup q\overline{H_{s,t}}.$
If $q\leq 2(\beta-1)$ $(<t)$, then $\mathbb{H}_{<t}=\varnothing$, and thus $G^\star-K\cong (p-q)K_t \cup q\overline{H_{s,t}}.$
If $q>2(\beta-1)$. then by Lemma \ref{lem4.2}, we have $G^\star-K\cong pK_t\cup K_q$.

Now assume that $q=2$. Since $\beta\geq2$, we have $q\leq 2(\beta-1)$ and thus $\mathbb{H}_{<t}=\varnothing$.
If $\beta>2$, then $\mathbb{H}_{t+2}=\varnothing$ by Theorem \ref{thm3.2}.
Thus, $G^\star-K\cong(p-q)K_t\cup q\overline{H_{s,t}}.$
It remains the case $q=\beta=2$.
Now if $\mathbb{H}_{t+1}\neq\varnothing$, then $\mathbb{H}_{t+2}=\varnothing$ by Lemma \ref{lem4.3}.
This implies that $G^\star-K\cong(p-2)K_t\cup\overline{H_{s,t}}\cup\overline{H_{s,t}}.$
Recall that $e\left(\overline{H_{s,t}}\right)={t\choose 2}+\beta-1={t\choose 2}+1.$
Now let $H'=K_t\cup S^1\left(\overline{H_{s,t}}\right).$
Then $|H'|=2|\overline{H_{s,t}}|=2t+2$ and
$$e(H')=e(K_t)+e\left(\overline{H_{s,t}}\right)+1=2e\left(\overline{H_{s,t}}\right).$$
Since $S^1\left(\overline{H_{s,t}}\right)$ is a subdivision of $\overline{H_{s,t}}$,
we can easily see that $\pi\left(\overline{H_{s,t}}\cup\overline{H_{s,t}}\right)\prec \pi(H'),$
a contradiction. Thus, $\mathbb{H}_{t+1}=\varnothing$. It follows that $|\mathbb{H}_{t+2}|=1.$
From Theorem \ref{thm3.2}, we have $G^\star-K\cong(p-1)K_t\cup S^1\left(\overline{H_{s,t}}\right).$
This completes the proof.
\end{proof}

Combining with Theorems \ref{thm4.1} and \ref{thm4.2}, we completes the proof of Theorem \ref{thm3}.

\section{Proof of Theorems \ref{thm4} and \ref{thm5}}

~~~~Throughout this section,
let $G^*$ be the extremal graph with the maximum spectral radius
over all $n$-vertex connected $K_{1,t}$-minor free graphs, $\rho=\rho(G^*)$
and $X=(x_1,x_2,\ldots,x_n)^T$ be the Perron vector of $G^*$.
Furthermore, assume that $u^*\in V(G^*)$ with $x_{u^*}=\max_{u\in V(G^*)}x_u$,
and let $A=N_{G^*}(u^*),$ $B=V(G^*)\setminus(A\cup \{u^*\})$, $N^2(u^*)=\{d_{G^*}(w,u^*)=2|w\in B\}.$
Above all, we need a lemma.

\begin{lem}\cite{WXH}\label{le000}
Let $G$ be a connected graph and $X=(x_1,x_2,\ldots,x_n)^T$ be its Perron vector.
If $x_u\geq x_v$ and $N_G(v)\setminus (N_G(u)\cup\{u\})=S\neq \varnothing$ for some $u,v\in V(G)$,
then $\rho\left(G-\{vw|w\in S\}+\{uw|w\in S\}\right)>\rho(G)$.
\end{lem}

If $n\le t$, it is clear that $G^*\cong K_n$.
Moreover, there exists no connected $K_{1,2}$-minor free graph of order $n\geq3$.
Therefore, we need only consider connected $K_{1,t}$-minor free graphs for $t\geq3$ and $n\geq t+1$.

\begin{lem}\label{le001}
$\rho>t-2$ and $|A|=t-1$.
\end{lem}

\begin{proof}
Let $K_t-e$ be the graph obtained from $K_t$ by deleting an edge,
and recall that $S^{n-t}(K_t)$ is obtained from $K_t$ by subdividing $n-t$ times of one edge.
Clearly, $S^{n-t}(K_t)$ is $K_{1,t}$-minor free and contains $K_t-e$ as a proper subgraph.
It follows that $\rho(G^*)\geq \rho(S^{n-t}(K_t))>\rho(K_t-e)$.
Let $\rho'=\rho(K_t-e)$ and $Y=(y_1,\ldots,y_1,y_2,y_2)^T$ be the Perron vector of $K_t-e$,
where $y_1$ corresponds to the $t-2$ vertices of degree $t-1$ and $y_2$
corresponds to the two vertices of degree $t-2$.
Then we have
$$\rho' y_1=(t-3)y_1+2y_2,~~~~ \rho' y_2=(t-2)y_1.$$
Solving these two equations, we have $\rho'^2-(t-3)\rho'-2(t-2)=0$.
Since $\rho=\rho(G^*)>\rho(K_t-e)$, we have
\begin{align}\label{a001}
\rho^2-(t-3)\rho>2(t-2).
\end{align}
Clearly, $\rho>t+2$ for $t\geq3$, as claimed.
Furthermore, $\rho x_{u^*}=\sum_{v\in A}x_v\le |A|x_{u^*}.$
It follows that $|A|\ge \rho>t-2$.
On the other hand,
we see that $|A|\leq\Delta(G^*)\le t-1$ for forbidding $K_{1,t}$-minor.
Therefore, $|A|=t-1$.
\end{proof}

By Lemma \ref{le001}, $d_A(v)\leq t-2$ for any $v\in A$.
Let $A_0=\{v\in A|d_A(v)=t-2\}$ and $A_1=A\setminus A_0$.
Clearly, $d_B(v)=0$ for any $v\in A_0$, since $\Delta(G^*)=t-1$.

\begin{lem}\label{le002}

\noindent(i) $d_B(v)\le 1$ for any $v\in A_1$.

\noindent(ii) $d_B(w)\le 2$ for any $w\in B$.
Particularly, $d_B(w)\le 1$ for $w\in N^2(u^*)$.
\end{lem}

\begin{proof}
(i) Suppose that $d_B(v_0)\geq2$ for some $v_0\in A_1$.
Then $G^*$ contains a double star with a non-pendant edge $u^*v_0$ and $|A|-1+d_B(v_0)$ leaves.
Since $|A|-1+d_B(v_0)\geq t$, we find a $K_t$-minor in $G^*$, a contradiction.
So the claim holds.

(ii) For any vertex $w\in B$, we can find a shortest path $P$ from $w$ to vertices in $A$.
Clearly, $d_{B\cap V(P)}(w)\leq1$, and particularly, $d_{B\cap V(P)}(w)=0$ if $w\in N^2(u^*)$.
Let $v\in A$ be the other endpoint of $P$.
Then $G^*$ contains a tree consisting of a path $P+vu^*$ and $|A|-1+d_{B\setminus V(P)}(w)$ leaves.
Since $G^*$ is $K_{1,t}$-minor free, we have $|A|-1+d_{B\setminus V(P)}(w)\leq t-1$.
It follows that $d_B(w)\leq d_{B\cap V(P)}(w)+1$, and thus the claim holds.
\end{proof}

\begin{lem}\label{le003}
Let $v_1\in A_1$ and $w\in B$ such that $v_1w$ is an edge,
and $\overline{N}_A(v_1)$ be the set of non-neighbors of $v_1$ in $A$.
Then $x_w\le \sum_{v\in \overline{N}_A(v_1)}x_v.$
\end{lem}

\begin{proof}
By Lemma \ref{le002}, $w$ is the unique neighbor of $v_1$ in $B$.
Thus we have
$$\rho x_{u^*}=\sum\limits_{v\in A\setminus \{v_1\}}x_v+x_{v_1},
~~~\rho x_{v_1}=\sum\limits_{v\in N_A(v_1)}x_v+x_{w}+x_{u^*}.$$
It follows that
$\rho (x_{u^*}-x_{v_1})=(x_{v_1}-x_{u^*})-x_w+\sum\limits_{v\in \overline{N}_A(v_1)}x_v.$
Therefore,
$$ x_w= (\rho+1)(x_{v_1}-x_{u^*})+\sum\limits_{v\in \overline{N}_A(v_1)}x_v \le \sum\limits_{v\in \overline{N}_A(v_1)}x_v,$$
since $x_{v_1}\le x_{u^*}$.
\end{proof}

\begin{thm}\label{t002}
If $n=t+1$, then $\overline{G^*}\cong\frac n2K_{1,1}$ for even $n$,
and $\overline{G^*}\cong K_{1,2}\cup\frac {n-3}2K_{1,1}$ for odd $n$.
In other words, $\overline{G^*}\cong H_{1,t}$.
\end{thm}

\begin{proof}
Recall that $|A|=t-1$. Then $|B|=1$. Say $B=\{w\}$, then $N_{G^*}(w)\subseteq A_1$.
We can see that $N_{G^*}(w)=A_1$. Indeed,
if there exists a vertex $v\in A_1$ with $vw\notin E(G^*)$,
then $\max\{d_{G^*}(v),d_{G^*}(w)\}\leq t-2$.
So, $\Delta(G^*+vw)=\Delta(G^*)=t-1$, and thus, $G^*+vw$ also contains no $K_{1,t}$.
However, $\rho(G^*+vw)>\rho (G^*)$, a contradiction.

Recall that $A_1=\{v\in A |d_A(v)\leq t-3\}$.
We now claim that $d_A(v)=t-3$ for any $v\in A_1$.
Suppose that $d_A(v_0)\le t-4$ for some $v_0\in A_1$ and let $v_1\in \overline{N}_A(v_0)$.
Then $d_A(v_1)\leq t-3$, and thus $v_1\in A_1=N_{G^*}(w)$.
Furthermore, $d_A(v_1)=t-3$ (Otherwise,
$\Delta(G^*+v_0v_1)=\Delta(G^*)$ and $\rho(G^*+v_0v_1)>\rho(G^*)$).
By Corollary $\ref{le004}$, we have $x_{w}\le x_{v_0}$.
Now let $G'=G^*-\{wv_1\}+\{v_0v_1\}$.
Then $\Delta(G')=\Delta(G^*)$, and since $v_0\in A_1=N_{G^*}(w)$, $G'$ is still a connected graph.
However, by Lemma \ref{le000}, we get $\rho(G')>\rho(G^*)$, a contradiction.
So, $d_A(v)=t-3$ for any $v\in A_1$.
Note that each vertex in $A_0$ dominates $A$.
The above claim implies that
$G^*[A_1]$ is isomorphic to $K_{|A_1|}$ by deleting a perfect matching.
Hence, $|A_1|$ is even.

It is known that $\rho(G)\leq \Delta(G)$ for any connected graph $G$,
with equality if and only if $G$ is $\Delta$-regular.
Now, $G^*$ is $t-1$-regular if and only if $A_1=A$.
It follows that if $t+1$ is even, then $A_1=A$, and thus $\overline{G^*}$ is
the union of $\frac{t+1}2$ independent edges, that is, $u^*w$ and others within $A$.

If $t+1$ is odd, then $|A|=t-1$ is also odd, and since $|A_1|$ is even, we have $|A_1|<|A|$
(see Fig. \ref{fig001}).
By symmetry, $x_v=x_{u^*}$ for any $v\in A_0$,
and we may let $x_v=x_1$ for any $v\in A_1$.
Thus, $\rho x_w=|A_1|x_1,$ $\rho x_{u^*}=|A_1|x_1+|A_0|x_{u^*}$
and $$\rho x_1=(|A_1|-2)x_1+(|A_0|+1)x_{u^*}+x_w.$$
Note that $|A_0|+|A_1|=|A|=t-1$. Solving above equations, we get
\begin{align}\label{a002}
\rho^3-(t-3)\rho^2-(2t-2)\rho+|A_0||A_1|=0.
\end{align}
Since $|A_1|$ is even and $t+1$ is odd, we have $2\le |A_1|\le t-2$ and $t\geq4$.
Thus, $$|A_0||A_1|\ge \min\{2(t-3), t-2\}=t-2,$$
with $|A_0||A_1|=t-2$ if and only if $|A_1|=t-2$.
Since $\rho$ is the maximum spectral radius,
by (\ref{a002}) we can see that $|A_0||A_1|$ must attain its minimum.
Therefore, $|A_1|=t-2$ and thus $|A_0|=1$.
Now, $\overline{G^*}$ is the union of $\frac{n-3}2$ independent edges with all endpoints in $A_1$,
and a path $u^*wv$ with $v\in A_0$.
This completes the proof.
\end{proof}

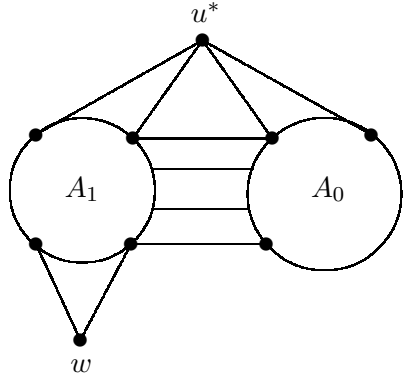
\begin{figure}[htp]
 \setlength{\unitlength}{1pt}
\begin{center}
\begin{picture}(146.5,95.9)
\put(72.5,83.4){\circle*{5}} \put(98.6,46.4){\circle*{5}}
\qbezier(72.5,83.4)(85.6,64.9)(98.6,46.4)
\put(10.2,47.9){\circle*{5}}
\qbezier(72.5,83.4)(41.3,65.6)(10.2,47.9)
\put(135.6,47.9){\circle*{5}}
\qbezier(72.5,83.4)(104.0,65.6)(135.6,47.9)
\put(26.8,-29.7){\circle*{5}} \put(10.2,6.5){\circle*{5}}
\qbezier(26.8,-29.7)(18.5,-11.6)(10.2,6.5)
\put(45.7,6.5){\circle*{5}}
\qbezier(26.8,-29.7)(36.3,-11.6)(45.7,6.5)
\put(46.4,46.4){\circle*{5}}
\qbezier(98.6,46.4)(72.5,46.4)(46.4,46.4)
\qbezier(72.5,83.4)(59.5,64.9)(46.4,46.4)
\put(21.0,31.9){\makebox(0,0)[tl]{$A_{1}$}}
\put(113.1,31.9){\makebox(0,0)[tl]{$A_{0}$}}
\put(68.2,98){\makebox(0,0)[tl]{$u^*$}}
\put(23.2,-37.0){\makebox(0,0)[tl]{$w$}}
\qbezier(54.7,26.8)(54.7,15.6)(46.8,7.6)
\qbezier(46.8,7.6)(38.8,-0.4)(27.6,-0.4)
\qbezier(27.6,-0.4)(16.3,-0.4)(8.3,7.6)
\qbezier(8.3,7.6)(0.4,15.6)(0.4,26.8)\qbezier(0.4,26.8)(0.4,38.1)(8.3,46.0)
\qbezier(8.3,46.0)(16.3,54.0)(27.5,54.0)\qbezier(27.5,54.0)(38.8,54.0)(46.8,46.0)
\qbezier(46.8,46.0)(54.7,38.1)(54.7,26.8)
\qbezier(146.9,25.4)(146.9,13.5)(138.5,5.1)\qbezier(138.5,5.1)(130.1,-3.4)(118.2,-3.4)
\qbezier(118.2,-3.4)(106.3,-3.4)(97.9,5.1)\qbezier(97.9,5.1)(89.4,13.5)(89.4,25.4)
\qbezier(89.4,25.4)(89.4,37.3)(97.9,45.7)\qbezier(97.9,45.7)(106.3,54.1)(118.2,54.1)
\qbezier(118.2,54.1)(130.1,54.1)(138.5,45.7)\qbezier(138.5,45.7)(146.9,37.3)(146.9,25.4)
\put(96.4,6.5){\circle*{5}} \qbezier(45.7,6.5)(71.1,6.5)(96.4,6.5)
\qbezier(53.7,34.8)(72.5,34.8)(91.4,34.8)
\qbezier(54.4,19.6)(72.1,19.6)(89.9,19.6)
\end{picture}
\end{center}\vspace{1.2cm}
\caption{The extremal graph $G^*$ for odd $n=t+1$.}\label{fig001}
\end{figure}

Now we consider the case $n\geq t+2$.
Recall that $S^{k}(G)$ denotes a graph obtained from $G$ by subdividing $k$ times of one edge.

\begin{thm}\label{t003}
If $n\ge t+2$, then $G^*\cong S^{n-t}(K_t)$.
\end{thm}

\begin{proof}
Observe that a graph $G$ is $K_{1,3}$-minor free if and only if $\Delta(G)\leq2$.
Therefore, if $t=3$, then $G^*\cong C_n$, in other words, $G^*\cong S^{n-3}(K_3)$.
Next, let $t\geq4$. The proof is divided into several claims.

\begin{claim}\label{cl001}
$|A_1|\geq|N^2(u^*)|\geq2$.
\end{claim}

\begin{proof}
Since $|A|=t-1$ and $n\geq t+2$,
we have $N^2(u^*)\neq\varnothing$.
We first show $|N^2(u^*)|\geq2$. Suppose to the contrary that $N^2(u^*)=\{w_0\}$.
By Lemma \ref{le002}, we have $d_B(w_0)\le 1$,
and $d_B(w)\le 2$ for each $w\in B$.
Since $G^*$ is connected, $G^*[B]$ is a pendant path $P=w_0w_1\cdots w_{n-t-1}$.
Take an arbitrary vertex $v_0\in N_A(w_0)$.
If $\overline{N}_A(w_0)\subseteq N_{G^*}(v_0)$, then
$|N_{G^*}(v_0)\cup N_{G^*}(w_0)|=|A\cup \{u^*,w_0,w_1\}|=t+2$.
Hence, $G^*$ contains a double star with a non-pendant edge $v_0w_0$ and $t$ leaves,
a contradiction.
Thus, there exists a vertex $v_1\in \overline{N}_A(w_0)\cap \overline{N}_{G^*}(v_0)$.
Then $d_{G^*}(v_1)\leq t-2$.

Now, let $G'=G^*+\{v_1w_{n-t+1}\}$ and $G''$ be the graph obtained from $G'$
by contracting the path $w_0w_1w_2\cdots w_{n-t-1}v_1$ into an edge $w_0v_1$.
Clearly, $|G''|=t+1$ and $\Delta(G'')=\Delta(G')=\Delta(G^*)=t-1$.
Therefore, $G''$ is $K_{1,t}$-minor free, and thus $G'$ is too.
However, $\rho(G')>\rho(G^*)$, a contradiction.
So $|N^2(u^*)|\geq2$.
Furthermore, by Lemma \ref{le002} (i), we have $|A_1|\geq|N^2(u^*)|$.
\end{proof}

Now let $A_{11}=\{v\in A_1|d_A(v)\le t-4\}$.
Since $n\geq t+2$, we have $|B|\geq2$.
So we may assume that $w',w''\in B$ with $x_{w'}=\max\limits_{w\in B}x_w$
and $x_{w''}=\max\limits_{w\in B\setminus\{w'\}}x_w$.

\begin{claim}\label{cl002}
If $d_B(v_0)=0$ for some $v_0\in A_1$,
then $x_{w'}+x_{w''}\le x_{u^*}+\sum\limits_{v\in A_{11}}x_v.$
\end{claim}

\begin{proof}
We first show $w'\in N^2(u^*)$.
Otherwise, $w'\in B\setminus N^2(u^*)$.
By Lemma \ref{le002}, $d_B(w')\le 2$,
and hence, $\rho x_{w'}=\sum_{w\in N_B(w')}x_{w}\le 2x_{w'}$.
It follows that $\rho\leq 2$, which contradicts
Lemma \ref{le001} as $t\geq4$.
Since $w'\in N^2(u^*)$, we have $d_B(w')\leq1$ and
\begin{align}\label{a003}
\rho x_{w'}\le x_{w''}+\sum_{v\in N_A(w')}x_v.
\end{align}
If $w''\in B\setminus N^2(u^*)$,
then $d_B(w'')\leq2$ and hence
$\rho x_{w''}\le x_{w'}+x_{w''}$,
which implies \begin{align}\label{a003'}
x_{w''}\leq \frac{x_{w'}}{\rho-1}.
\end{align}
Combining with (\ref{a003}), we have
\begin{align}\label{a003''}
x_{w'}\leq \frac{\rho-1}{\rho^2-\rho-1}\sum\limits_{v\in N_A(w')}x_v.
\end{align}
Furthermore, by Claim \ref{cl001}, $|N^2(u^*)|\geq2$. Hence,
there exist a vertex $w\in N^2(u^*)\setminus\{w'\}$ and
a vertex $v$ in $N_A(w)\setminus N_A(w')$.
So, $d_A(w')\leq |A\setminus\{v,v_0\}|=t-3$.
Combining it with (\ref{a003'}) and (\ref{a003''}), we have
$$x_{w'}+x_{w''}\le \frac{\rho}{\rho-1}x_{w'}\leq
\frac{\rho}{\rho^2-\rho-1}\sum\limits_{v\in N_A(w')}x_v\leq \frac{(t-3)\rho}{\rho^2-\rho-1}x_{u^*}.$$
Note that (\ref{a001}) implies $\rho^2-\rho-1>(t-3)\rho$.
So, $x_{w'}+x_{w''}<x_{u^*}$, as desired.

Now assume that $w''\in N^2(u^*)$.
By Lemma \ref{le002}, $d_B(w'')\le 1$ and thus
$\rho x_{w''}\le x_{w'}+\sum_{v\in N_A(w'')}x_v$.
Combining with (\ref{a003}), we have
\begin{align}\label{a004}
\rho(x_{w'}+x_{w''})\leq (x_{w'}+x_{w''})+\sum_{v\in N_A(w')\cup N_A(w'')}x_v,
\end{align}
since $N_A(w')\cap N_A(w'')=\varnothing$ by Lemma \ref{le002}.
Moreover, $$|N_A(w')\cup N_A(w'')|\leq |A\setminus\{v_0\}|=t-2.$$
If $|N_A(w')\cup N_A(w'')|\leq t-3$, then by (\ref{a004}),
we have $x_{w'}+x_{w''}\leq \frac {t-3}{\rho-1}x_{u^*}<x_{u^*}$.
If $d_A(v_1)\leq t-4$ for some $v_1\in N_A(w')\cup N_A(w'')$,
then $v_1\in A_{11}$ and $$x_{w'}+x_{w''}\leq \frac {t-3}{\rho-1}x_{u^*}+x_{v_1}<x_{u^*}+\sum_{v\in A_{11}}x_v.$$

Next we may assume that $|N_A(w')\cup N_A(w'')|=t-2$
and $d_A(v)=t-3$ for any $v\in N_A(w')\cup N_A(w'')$.
Now $N_A(w')\cup N_A(w'')=A\setminus\{v_0\}$.
Since $v_0\in A_1$, $d_A(v_0)\leq t-3$ and thus $v_0$ has a non-neighbor $v_1\in N_A(w')\cup N_A(w'')$.
Therefore, $d_A(v_1)=t-3$, and $v_1$ dominates $N_A(w')\cup N_A(w'')$.
Let $v_1v_2$ be an edge between $N_A(w')$ and $N_A(w'')$.
Then $v_0$ is also a non-neighbor of $v_2$
(Otherwise, $|N_{G^*}(v_1)\cup N_{G^*}(v_2)|=|A\cup \{u^*,w',w''\}|=t+2$
and we get a $K_{1,t}$-minor).
Now we see that $d_A(v_0)\leq t-4$ and thus $v_0\in A_{11}$.
On the other hand, by Lemma \ref{le003}, we have $x_{w''}\leq x_{v_0}$.
It follows that $x_{w'}+x_{w''}\leq x_{u^*}+\sum_{v\in A_{11}}x_v.$
\end{proof}

\begin{claim}\label{cl003}
$d_B(v)=1$ for each vertex $v\in A_1$, and thus $e(A,B)=|A_1|$.
\end{claim}

\begin{proof}
By Lemma \ref{le002}, $d_B(v)\le 1$
for each vertex $v\in A_1$.
Suppose that $d_B(v_0)=0$ for some $v_0\in A_1$.
We can see that $\rho x_{u^*}=\sum\limits_{v\in A}x_v$ and
\begin{align}\label{a005}
\rho^2 x_{u^*}=\sum\limits_{v\in A} \sum\limits_{u\in N_{G^*}(v)}x_u
=|A|x_{u^*}+\sum\limits_{v\in A}d_A(v)x_v+\sum\limits_{w\in B}d_A(w)x_w.
\end{align}
From (\ref{a005}) and the definitions of $A_0$ and $A_{11}$, we have
\begin{eqnarray}\label{a006}
(\rho^2-(t-3)\rho)x_{u^*}
&=& \rho^2x_{u^*}-(t-3)\sum\limits_{v\in A}x_v  \nonumber\\
&=& |A|x_{u^*}+\sum\limits_{v\in A}(d_A(v)-t+3)x_v+\sum\limits_{w\in B}d_A(w)x_w \nonumber\\
&\le& |A|x_{u^*}+\sum\limits_{v\in A_{0}}x_v-\sum\limits_{v\in A_{11}}x_v+\sum\limits_{w\in B}d_A(w)x_w\nonumber\\
&\le& (|A|+|A_0|+e(A,B)-2)x_{u^*}+x_{w'}+x_{w''}-\sum\limits_{v\in A_{11}}x_v.
\end{eqnarray}
Combining Claim \ref{cl002} with $(\ref{a006})$, we have
$$(\rho^2-(t-3)\rho)x_{u^*}
\le (|A|+|A_0|+e(A,B)-1)x_{u^*}.
$$
Recall that $d_B(v)=0$ for any $v\in A_0$, and $d_B(v)\leq1$ for any $v\in A_1$. Therefore,
$$e(A,B)=e(A_1,B)=\sum\limits_{v\in A_1\setminus \{v_0\}}d_B(v)\le |A_1\setminus \{v_0\}|=|A_1|-1.$$
It follows that $\rho^2-(t-3)\rho\le 2(|A|-1)=2(t-2),$
which contradicts (\ref{a001}).
So, $d_B(v)=1$ for any $v\in A_1$ and thus $e(A,B)=e(A_1,B)=|A_1|$.
\end{proof}

\begin{claim}\label{cl004}
$|A_{11}|\le 2$.
\end{claim}

\begin{proof}
Suppose to the contrary that $|A_{11}|\ge 3$.
Then $|A\setminus A_{11}|\le t-4$,
and thus $$\rho x_{u^*}=\sum\limits_{v\in A_{11}}x_v
+\sum\limits_{v\in A\setminus A_{11}}x_v\le \sum\limits_{x\in A_{11}}x_v+(t-4)x_{u^*}.$$
Therefore, by Lemma \ref{le001}, we have
\begin{align}\label{a007}
 \sum\limits_{v\in A_{11}}x_v\ge (\rho-t+4)x_{u^*}\ge 2x_{u^*}.
\end{align}
Combining with (\ref{a006}) and (\ref{a007}), we have
\begin{eqnarray*}
  (\rho^2-(t-3)\rho)x_{u^*}
  &\le& |A|x_{u^*}+ \sum\limits_{v\in A_{0}}x_v-\sum\limits_{v\in A_{11}}x_v+\sum\limits_{w\in B}d_A(w)x_w        \\
  &\le&   (|A|+|A_0|-2)x_{u^*}+e(A,B)x_{u^*}.
\end{eqnarray*}
By Claim \ref{cl003}, we have $e(A,B)=|A_1|$.
Hence, $$\rho^2-(t-3)\rho\le (|A|+|A_0|+|A_1|-2)=2|A|-2=2(t-2),$$
which contradicts (\ref{a001}).
Thus, $|A_{11}|\le 2$, as claimed.
\end{proof}

\begin{claim}\label{cl005}
If $|A_1|\ge 3$, then $A_{10}$ is a clique, where $A_{10}=\{v\in A|d_A(v)=t-3\}$.
\end{claim}

\begin{proof}
Recall that $A_{11}=\{v\in A|d_A(v)\leq t-4\}$. Then $A_{10}=A_1\setminus A_{11}.$
Suppose to the contrary that there exist two non-adjacent vertices $v_1,v_2\in A_{10}$.
Since $d_A(v_1)=d_A(v_2)=t-3$, both $v_1$ and $v_2$ dominate $A\setminus \{v_1,v_2\}$.
Since $|A_1|\ge 3$, we can take a vertex $v_3\in A_1\setminus \{v_1,v_2\}$.
Furthermore, by Claim \ref{cl003}, there exists a unique vertex $w_i\in N_B(v_i)$ for $i\in\{1,2,3\}$.
If $w_3\not=w_1$, then $$|N_{G^*}(v_1)\cup N_{G^*}(v_3)|\ge |A\cup \{u^*,w_1,w_3\}|=t+2.$$
If $w_3\not=w_2$, then we similarly have $|N_{G^*}(v_2)\cup N_{G^*}(v_3)|\geq t+2.$
If $w_1=w_2=w_3$, then by Claim \ref{cl001}, $|N^2(u^*)|\ge 2$
and thus there exist a vertex $w_4\in N^2(u^*)\setminus \{w_1\}$ and a vertex $v_4\in N_A(w_4)$.
It follows that $$|N_{G^*}(v_1)\cup N_{G^*}(v_4)|\ge |A\cup \{u^*,w_1,w_4\}|=t+2.$$
Thus we always find a $K_{1,t}$-minor, a contradiction.
\end{proof}

Now we are ready to give the final proof of the theorem.
By Claim \ref{cl001}, $|A_1|\geq|N^2(u^*)|\geq2$.
If $|A_1|= 2$, then $|N^2(u^*)|=2$, and $G^*[A\cup\{u^*\}]\cong K_t-e$,
where the unique non-edge lies in $A_1$.
Now we can see that $d_A(w)=1$ and $d_B(w)\leq1$ for each $w\in N^2(u^*)$.
Moreover, $d_B(w)\leq 2$ for any $w\in B\setminus N^2(u^*)$.
Since $G^*$ is an extremal graph, we can observe that
$G^*[A_1\cup B]$ is a path with both endpoints in $A_1$.
This implies that $G^*\cong S^{n-t}(K_t)$, as desired.

It remains the case $|A_1|\ge 3$. We shall prove that this is impossible.
By Claim \ref{cl004}, $|A_{11}|\le 2$. Suppose that $|A_{11}|=|\{v_1,v_2\}|=2$.
For $i\in\{1,2\}$, since $d_A(v_i)\le t-4$, $v_i$ has at least one non-neighbor $v'_i\in A_{10}$.
By the definition of $A_{10}$, $d_A(v'_1)=d_A(v'_2)=t-3$, and hence $v'_1\not=v'_2$.
By Claim \ref{cl003}, there exists a unique vertex $w_i\in N_B(v'_i)$ for $i\in\{1,2\}$ (Possibly, $w_1=w_2$).
And by Lemma \ref{le003}, $x_{w_i}\le x_{v_i}$ for $i\in\{1,2\}$.
Thus, $x_{w_1}+x_{w_2}\le x_{v_1}+x_{v_2}=\sum_{v\in A_{11}}x_v$.
Recall that $e(A,B)=|A_1|$. Then by $(\ref{a006})$,
\begin{eqnarray*}
(\rho^2-(t-3)\rho)x_{u^*}
&\le & (|A|+|A_0|+e(A,B)-2)x_{u^*}+x_{w_1}+x_{w_2}-\sum\limits_{v\in A_{11}}x_v\nonumber\\
&\leq& (|A|+|A_0|+|A_1|-2)x_{u^*} \nonumber\\
&=& 2(t-2))x_{u^*},
\end{eqnarray*}
a contradiction with (\ref{a001}). Thus $|A_{11}|\le 1$.

According to Claim \ref{cl005}, $A_{10}$ is a clique.
Recall that $A_{10}=\{v\in A|d_A(v)=t-3\}$.
Therefore, each vertex $v\in A_{10}$ has a non-neighbor $v_0\in A_{11}$.
And since $|A_{11}|\leq1$, we have that $A_{11}=\{v_0\}$ and $e(\{v_0\},A_{10})=0$.
Let $A_{10}=\{v_1,v_2,\ldots,v_{|A_{10}|}\}$.
By Claim \ref{cl003}, there exists a unique vertex $w_i\in N_B(v_i)$ for each $v_i\in A_1$
(Possibly, $w_i=w_j$ for some $i\neq j$).
Since $|N^2(u^*)|\geq2$ and $G^*[B]$ consists of disjoint paths,
we may assume without loss of generality that

\vspace{0.1cm}
(i) $w_{|A_{10}|}\not=w_0$;

\vspace{0.1cm}
(ii) if $|N^2(u^*)|\geq3$, then $w_0$ and $w_{|A_{10}|}$ belong to two distinct paths of $G^*[B]$.

\vspace{0.1cm}
\noindent Note that $|A_1|\ge 3$, then $|A_{10}|=|A_1|-|A_{11}|\ge 2$.
Let $$G'=G^*-\{v_iw_i|i=1,2,\ldots,|A_{10}|-1\}+\{v_iv_0|i=1,2,\ldots,|A_{10}|-1\}.$$
By Lemma \ref{le003}, $x_{w_i}\le x_{v_0}$ for each $w_i\in A_{10}$.
Furthermore, by Lemma \ref{le000} we have $\rho(G')>\rho(G^*)$.
We can observe that $G'[A\cup\{u^*\}]\cong K_t-e$, where the unique non-edge is $v_0v_{|A_{10}|}$.
Moreover, $G'[B]$ is still the disjoint union of paths,
and the only edges between $A$ and $B$ are $v_0w_0$ and $v_{|A_{10}|}w_{|A_{10}|}$.
These observations imply that $G'$ is a subgraph of $S^{n-t}(K_t)$.
It follows that $\rho(S^{n-t}(K_t))\ge \rho(G')>\rho(G^*)$, a contradiction.
This completes the proof.
\end{proof}

Recall that $H_{1,t}$ is a star forest of order $t+1$, precisely,
the disjoint union of $\lfloor\frac{t+1}{2}\rfloor$ stars
in which all but at most one are isomorphic to $K_{1,1}$.
By Theorems \ref{t002} and \ref{t003}, we complete the proof of Theorem \ref{thm4}.
Furthermore, we know that for any connected graph $G$, $\rho(G)\leq \Delta(G)$,
with equality if and only if $G$ is a $\Delta$-regular graph.
One can see that Theorem \ref{thm5} is a direct corollary of Theorem \ref{thm4}.

\end{document}